\documentclass[a4paper,reqno,12pt]{amsart}
\usepackage{amsfonts,amssymb,latexsym,xspace,epsfig,graphics,color}
\usepackage{amsmath,enumerate,stmaryrd,xy}
\usepackage{bbm}
\usepackage{natbib}
\numberwithin{figure}{section}
\usepackage{subfigure}
\usepackage{dsfont}
\usepackage{a4wide}
\usepackage{enumitem} 

\usepackage{hyperref}
\hypersetup{ colorlinks,
linkcolor=blue,
filecolor=green,
urlcolor=blue,
citecolor=blue }

\usepackage{stmaryrd}
\newtheorem{theorem}{Theorem}[section]       
\newtheorem{prop}[theorem]{Proposition} \newtheorem{lemma}[theorem]{Lemma}
\newtheorem{corollary}[theorem]{Corollary}
\newtheorem{defn}[theorem]{Definition}
\newtheorem{remark}[theorem]{Remark}
\newtheorem{exa}[theorem]{Example}

\def\bbK{{\mathbb K}} 
\def\A{\mathcal{A}}
\def\T{\mathcal{T}}

\def\A{\mathcal{A}}

\def\supp {\mathop {\rm supp}\nolimits}

\def\aut {\mathop {\rm Aut}\nolimits}
\def\Der {\mathop {\rm Der}\nolimits}
\newcommand{\Dim}{{\rm dim}}

\newcommand{\ann}{{\rm ann}}
\newcommand{\spa}{{\rm span}}
\def\Der {\mathop {\rm Der}\nolimits}
\newcommand{\rk}{{\rm rk}}
\def\T {\mathcal{T}}
\def\w{\omega}
\newcommand{\bK}{\mathbb{K}}

\newcommand{\loops}{{\rm L}}

\newcommand{\nloops}{{\rm NL}}

\def\diag{\mathop{\hbox{\rm diag}}}

\bibdata{prob}
\bibstyle{alpha} 

\title[]{Derivations and loops of some evolution algebras.}

\author[Yolanda Cabrera Casado]{Yolanda Cabrera Casado}
\address{Yolanda Cabrera Casado: Departamento de Matem\'atica Aplicada, E.T.S. Ingenier\'\i a Inform\'atica, Universidad de M\'alaga, Campus de Teatinos s/n. 29071 M\'alaga.   Spain. }
\email{yolandacc@uma.es}

\author[Paula Cadavid]{Paula Cadavid}
\address{Paula Cadavid:  Centro de Matemática, Computação e Cognição,  Universidade Federal do ABC, Avenida dos Estados, 5001-Bangu - Santo Andr\'e - SP, Brazil.}
\email{paula.cadavid@ufabc.edu.br, pacadavid@gmail.com}

\author[Tiago Reis]{Tiago Reis}
\address{Tiago Reis: Universidade Tegnol\'ogica  Federal do Paran\'a, Av. Alberto Carazzai, 1640 - Campus Corn\'elio Proc\'opio - PR, Brazil and Universidade Federal do ABC, Avenida dos Estados, 5001 - Bangu - Santo Andr\'e - SP, Brazil.}
\email{treis@utfpr.edu.br }

\thanks{ The  first author is supported by the Spanish Ministerio de Ciencia e Innovaci\'on through project  PID2019-104236GB-I00 and  by the Junta de Andaluc\'{\i}a  through projects  FQM-336 and UMA18-FEDERJA-119,  all of them with FEDER funds. 
}


\begin{document}

\maketitle
\date{}

\begin{abstract} 
In this work we study the space of derivations of non-degenerate evolution algebras. We improve some results obtained recently in the literature and, as a consequence, we advance in the description of the derivations for $n$-dimensional Volterra evolution algebras. In addition, we introduce the notion of loop of an evolution algebra and we analyze under which conditions the set of loops is invariant under change of basis. 

\end{abstract}

\section{Introduction}

\subsection{Evolution algebras ans its derivations}
The evolution algebras are non-associative algebras introduced in \cite{tian1} by Tian and Vojtechovsky in 2006, who established the theoretical foundations of these structures. In addition to emerging to model non-Mendelian genetic, in \cite{tian} Tian identified a whole series of connections with other areas such as graph theory, group theory and discrete-time Markov chains, among others. For a recent review of the advances in this type of algebras we refer the reader to \cite{ceballos}. 

In this work we are interested in studying the space of derivations of some evolution algebras. We point out that although there are many works describing partially such a space, by using different approaches, a complete characterization is still an unfinished task. For any evolution algebra, \cite{tian} described the derivations in terms of a system of equations which becomes the starting point for the characterization of the derivations for different families of evolution algebras.
The study of the derivations of evolution algebras with non-singular structure matrices was done in \cite{COT} for complex algebras and extended in \cite{Elduque/Labra/2019} for algebras over a field with any characteristic.  While \cite{cardoso} gave a complete characterization for the space of derivations of two-dimensional evolution algebras, \cite{Alsarayreh/Qaralleh/Ahmad/2017} studied the derivations of certain three-dimensional evolution algebras (solvable and nilpotent). Later,  \cite{Qaralleh/Mukhamedov/2019}
 provided a description of the derivations of three dimensional Volterra evolution algebras. In \cite{PMP3} the authors provided a characterization for the case of evolution algebras associated to graphs over a field of zero characteristic, which was generalized later for fields of any characteristic in \cite{Reis/Cadavid}. The novelty in the approach developed in \cite{PMP3,Reis/Cadavid} rely on the connection between the set of equation mentioned above and the structural properties of the considered graph. Such an approach was explored in \cite{YPMP}, where the authors studied the space of derivations of some non-degenerate irreducible evolution algebras depending on the twin partition of an associated directed graph.

One of the contributions of our work is to provide a characterization, in a sense to be defined later, of the space of derivations of non-degenerate evolution algebras. We improve some results obtained recently in the literature and, as a consequence, we advance in the description of the derivations for $n$-dimensional Volterra evolution algebras.

 \subsection{Loops of an evolution algebra}

Our approach to deal with derivations is inspired in a combination of arguments developed by \cite{YPMP} and \cite{YKM}. While the former explores the structure of a directed graph associated to the evolution algebra, the last rely on a partition of the considered basis. One of the peculiarities of evolution algebras is that they are not defined by identities, so their study usually follows a different strategy than the one associated to other non-associative algebras like Jordan, Lie or power-associative algebras. An usual approach to deal with an evolution algebra is to fix its natural basis. However, many properties are not invariant through the chosen basis. Some examples of this is the connectedness of the associated direct graph (see \cite[Example 2.5]{Elduque/Labra/2015}) or the skew-symmetry of the structure matrix (see Example \ref{volterra:change_basis}). Therefore an interesting task is to know which properties are invariant under the chosen basis. Motivated by this question we study the phenomenon showed in Example \ref{volterra:change_basis}, where we change the basis of a Volterra evolution algebra and as a consequence we verify that each element of the diagonal of the structure matrix remains equal to zero. We prove that this is true in general for Volterra evolution algebras. Moreover, it was trying to answer this question that we solve a more general problem; namely, when the number of zeros in the diagonal of a structure matrix of an evolution algebra is invariant under the change of natural basis. We point out that this problem has been addressed previously in \cite[Proposition 2.13]{YKM} for the case where the algebra is perfect.

The nonzero elements belonging to the diagonal of the structure matrix are what we call the loops of an evolution algebra. Although our results related to this part of the paper are of independent interest, we observe that knowing the loops of an evolution algebras was useful to the study of its derivations in \cite{YPMP}.

 \subsection{Organization of the paper}
 Now, we will show how this paper is organized. In Section \ref{preli} we introduce the preliminary definitions and notations. Taking into account \cite[Theorem 2.11]{boudi20} we define a natural decomposition of a natural basis of an evolution algebra which will be an important tool for later results. 
 
 We begin Section \ref{dernon} by establishing a characterization of space of derivations for a non-degenerate evolution algebra in Proposition \ref{PP}. As a consequence, the Corollary \ref{blockder} proves that the derivation is a block matrix, up to reordering. The Proposition \ref{prop34} gives a connection between the set of derivations  and the fact of having an unique natural basis (in the sense that whatever other natural basis can obtain by permutations or product by scalars). This fact generalizes the results of \cite[Theorem 2.1]{COT}, \cite[Theorem 4.1 item (1)]{Elduque/Labra/2019} and \cite[Theorem 1]{YPMP}. In the particular case of $\Dim(\A^2)=1$ and the product of any two square elements of basis is different from zero, we provide  necessary and sufficient conditions for a linear operator to be a derivation (Proposition \ref{dii}). In fact, we show in Corollary \ref{dimen1} that, up to reordering, the derivation is a skew-symmetric matrix. One of the requirements is related to the fact that the matrix of derivation is a diagonal matrix. For this reason, we ask, on the one hand when the derivation will have, under suitable conditions, some of the entries of the main diagonal equals and, on the other hand, the main diagonal null. Proposition \ref{diago1} and Theorem \ref{diagonalceros} answer to these questions, respectively. 
 
 The section \ref{volterrader} is devoted to study of derivations in the case of Volterra evolution algebras. In the same way as before, we show a characterization of the space of derivations for this specific case. Fixed an Volterra evolution algebra and verifying that the product of any two square of elements of basis is different from zero, the main result of this section (Theorem \ref{teo:dervolt}) provides a way of finding another Volterra evolution algebra with structure matrix diagonal and space of derivations the same as the Volterra evolution algebra original. As the structure matrix is diagonal, to calculate the set of derivations is equivalent to calculate the set of derivations over certain evolution ideals of the algebra (Corollary \ref{blockmatrix}). In fact, what we need is to find conditions that ensure that the main diagonal of derivations is null.  In Proposition \ref{prop:sumalpha3=0} we claim that the requirement imposed on the elements of the natural basis in the Theorem \ref{teo:dervolt} can be replaced by other property related to the structure constants. Particularly, Proposition \ref{prop:e^4diineq0} and Proposition \ref{prop:alpha3=0} give conditions under which there exists a derivation of a non-degenerate Volterra evolution algebra is not a diagonal matrix.
 
 In Section \ref{loops} we start by defining the loops of an evolution algebra and study when this set is invariant under change of natural basis. First, we will consider the set of no loops and we prove in Theorem \ref{ceros}  that if an element of a natural decomposition is contained within the set of no-loops then its corresponding element in another natural decomposition is also contained within it. Next, we will focus in the set of loops and we have that if an evolution algebra has not loops relative to a natural basis then it has not loops relative to any natural basis (Corollary \ref{coro:loopsperfect2}), as we have said before. Moreover, the Theorem \ref{theo:invariance_of_loops}  and Proposition \ref{B_i: pequeno} provide  convenient criteria in terms of the elements of the natural decomposition for the number of loops to be an invariant. By contrast, we also give in Theorem \ref{theo:notinvariance_of_loops} some conditions  to find a new natural basis such that the number of loops not stay constant. A summarizing of conditions for invariability of number of loops can be seen in Corollary \ref{recopi}.

\section{Preliminaries}\label{preli}

In what follows $\bbK$ will denote, unless we state otherwise, a field such that ${\rm char}(\bK)=0$. In order to state the first definitions let $\Lambda:=\{1,\ldots,n\}$. An $n$-dimensional $\bbK$-algebra $\A$ is called  \emph{evolution algebra} if it  admits a basis $B=\{e_i \}_{i \in \Lambda}$ such that $e_{i}e_{j}=0$ whenever $i \neq j$. A basis with this property is known as \emph{natural basis}. The scalars $\omega_{ij} \in \bbK$ such that \[e_i^2=\displaystyle\sum_{k \in \Lambda} \w_{ik} e_k\] are called the \emph{structure constants} of $\A$ relative to $B$ and the matrix $M_{B}=(\omega_{ik})$ is called the \emph{structure matrix} of $\A$ relative to $B$. When $\A=\A^2$ or equivalently when $M_B$ is invertible, it is said that $\A$ is \emph{perfect}.

If $u=\sum_{i\in \Lambda}\alpha_ie_i$ is an element of $\A$ then the \emph{support of} $u$ \emph{relative to} $B$ is defined  as $\supp_B(u):=\{i\in \Lambda:   \alpha_i \neq 0\}$. In general, if $X \subseteq \A$, we have that $\supp_B (X)= \cup_{x \in X} \; \supp_B (x)$. For $u=e_i^2$ support of $e_i^2$ is called the \emph{first-generation descendents} of $i$ relative to the natural basis $B$, i.e., $D^{1}(i)=\left\{k\in  \Lambda, \colon \, \omega _{ik}\neq 0\right\}.$ 
By analogy, given a subset $U\subseteq \Lambda $, we let $D^{1}(U):=\supp_B(W)$ where $W=\{e_i \in B \colon i \in U\}$. 
Similarly, we say that $j$ is a \emph{second-generation} descendent of $i$ whenever $j \in D^1(D^1(i))$. Therefore $$D^2(i)=\bigcup_{k \in D^1(i)} D^1(k).$$ 
By recurrence, we define the set of \emph{$mth-$generation descendents} of $i$ as $$D^m(i)=\bigcup_{k \in D^{m-1}(i)} D^1(k).$$ 
Finally, the set of descendents of $i$ is defined as
$$D(i)=\bigcup_{m \in \mathbb{N}} D^m(i).$$
An evolution algebra $\A$ is \emph{non-degenerate} if there is a natural basis $B$ such that $e_i^2 \neq 0$ for all $e_i \in B$. We remark that $\A$ is a non-degenerate evolution algebra if and only if $D^{1}(i)\neq \emptyset$ for all $i\in \Lambda$. By  \cite[Lemma 2.7]{Elduque/Labra/2015} this definition does not depend on the chosen natural basis since $\ann(\A)=\spa (\{e_i: e_i^2=0\})$ where $\ann(\A):=\{x\in \A: x\A=0 \}$. Therefore $\A$ is non-degenerate if and only if $\ann(\A)=0$.

On the other hand, an evolution algebra $\A$ is \emph{reducible} if there exist two nonzero ideals $I$ and $J$ of $\A$ such that $\A=I \oplus J$. In other case, it is called \emph{irreducible}.

 We follow the Definition 2.1 and Definition 2.8 in \cite{boudi20}.  An evolution algebra $\A$ has an \emph{unique} natural basis if the subgroup of $\aut_\bK(\A)$ such that map natural basis into natural basis is precisely $S_n\rtimes (\bK^\times)^n$. This group was depicted in \cite{CSV2} for $n=3$. On the other hand, it is said that $\A$ has  \emph{Property} \rm{(2LI)} if  for any two different vectors $e_i, e_j$ of a natural basis,  $\{e_i^2, e_j^2\}$ is linearly independent. Note that any perfect evolution algebra has the Property \rm{(2LI)} but the reciprocal is not true (see \cite[Example 2.9]{boudi20}).

 Let $V$ be a $\bK$-vector space and $S$ be a subset of $V$. We denote by $\spa(S)$ the vector subspace generated by $S$ and  $\rk(S)$ the \emph{rank of $S$}, that is, the dimension of $\spa(S)$ as vector space.
 
 We recall that $\A$ is a \emph{Volterra evolution algebra} if  there exists a natural basis $B$ of $\A$ such that $M_B$ is a skew-symmetric matrix. In this case we say that $\A$ is a \emph{ Volterra evolution algebra relative to $B$}. Since we are considering algebras over a field of zero characteristic, the matrix $M_B$ has null diagonal. This family of algebras was introduced in \cite{Qaralleh/Mukhamedov/2019} where the authors give a connection between this kind of algebras with the ergodicities of Volterra quadratic stochastic operators and, among other things, they show that these algebras are not nilpotent and they calculate its derivations for some cases.

Note that if $\A$ is a Volterra evolution algebra then is not true that for any natural basis $B$ the structure matrix  $M_{B}$ is skew-symmetric.

\begin{exa}\label{volterra:change_basis} \rm Let $\A$ be an evolution algebra, and  let $B=\{e_1, e_2,e_3\}$ be a natural basis such that 
\[M_{B}=\left(
\begin{array}{rrr}
0 & 1 & 0\\
-1 & 0 & 1\\
 0&-1 & 0 \\
\end{array} \right ).
\]
Therefore $\A$ is a Volterra evolution algebra. On the other hand, if $B^{\prime }=\{f_1,f_2,f_3\}$ is such that $f_1=2e_1+e_3$, $f_2=\frac{1}{2}e_1 + e_3$  and $f_3=e_2$, then $B^{\prime}$  is a natural basis of $\A$ such that 
\[M_{B^{\prime}}=\left(
\begin{array}{rrr}
0 & 0 & 3\\
0 & 0 & -\frac{3}{4}\\
-1 & 2 & 0 \\
\end{array} \right )
\]
is non skew-symmetric. 
\end{exa}

Now, we recall some basic definitions and notation for directed graphs. A \emph{directed graph} is a $4$-tuple $E=(E^0, E^1, s_E, r_E)$ where $E^0$, $E^1$ are sets and $s_E, r_E: E^1 \to E^0$ are maps. The elements of $E^0$ are called the \emph{vertices} of $E$ and the elements of $E^1$ the \emph{arrows} or \emph{directed edges} of $E$. For
$f\in E^1$ the vertices $r(f)$ and $s(f)$ are called the \emph{range} and the \emph{source} of $f$, respectively. If $E^0$ and $E^1$ are both finite we say  that $E$ is \emph{finite}. A vertex $v\in E^{0}$ is called \emph{sink} if it verifies that $s(f) \neq v$, for every $f\in E^1$.  
A \emph{path} or a \emph{path from $s(f_1)$ to $r(f_m)$} in  $E$, $\mu$, is a finite sequence of arrows $\mu=f_1\dots f_m$
such that $r(f_i)=s(f_{i+1})$ for $i\in\{1,\dots,(m-1)\}$. In this case we say that $m$ is the \emph{length} of the path $\mu$ and denote it by $\vert \mu \vert=m$. Let $\mu = f_1 \dots f_m$ be a path in $E$ with $ \vert\mu\vert =m\geq 1$. If
$v=s(\mu)=r(\mu)$, then $\mu$ is called a \emph{closed path based at $v$}.
If $\mu = f_1 \dots f_m$ is a closed path based at $v$ and $s(f_i)\neq s(f_j)$ for
every $i\neq j$, then $\mu$ is called a \emph{cycle based at} $v$ or simply a \emph{cycle}. A cycle of length $1$ will be said to be a \emph{loop}. Given a finite graph $E$, its \emph{adjacency matrix} is the matrix $A_{E}=(a_{ij})$
where $a_{ij}$ is the number of arrows from $i$ to $j$. A graph $E$ is said to satisfy Condition (Sing) if among two vertices of $E^0$
there is at most
one arrow. If exist a path from $i$ to $j$ in $E$, then we define the \emph{distance} from $i$ to $j$ as $\delta (i,j)=\min \{|\mu| , \mu \text{ is a path from }i \text{ to } j\}$.

There are several ways to associate a graph to an evolution algebra (see \cite{YMV,Elduque/Labra/2015}). We consider the directed graph described in \cite{YMV} as follows. Given a natural basis $B=\{e_i \}_{i\in \Lambda}$ of an evolution algebra $\A$ and its structure matrix $M_B=(\omega_{ij})\in  {M}_\Lambda(\mathbb K)$, consider the matrix
$P=(a_{ij})\in  {M}_\Lambda(\mathbb K)$ such that $a_{ij}=0$ if $\omega_{ij}=0$ and $a_{ij}=1$ if $\omega_{ij}\neq 0$. The \emph{graph associated to the evolution algebra} $\A$ (relative to the basis $B$), denoted by $E_\A^{B}$ (or simply by $E$ if the algebra $\A$ and the basis $B$ are understood) is the directed graph whose adjacency matrix is given by $P=(a_{ij})$. In this way, we only consider graphs satisfying  Condition Sing.

By analogy with graph theory we define the following notions. Let $\A$ be an evolution algebra with natural basis $B$ and let $i,j \in \Lambda$. We say that $i$ and $j$ are \emph{twins relative to $B$} if $D^1(i)=D^1(j)$. We notice that by defining the relation $\sim_{t_B}$ on the set of indices $\Lambda$ by $i\sim_{t_B} j$ whether $i$ and $j$ are twins relative to $B$, then $\sim_{t_B}$ is an equivalence relation. An equivalence class of the twin relation $\sim_{t_B}$ is referred to as a {\it twin class relative to $B$}. In other words, the twin class of an index $i$, that we will denote by $\mathcal{T}(i)$,  is the set $\mathcal{T}(i):=\{j\in \Lambda :i \sim_{t_B} j\}$. The set of all twin classes relative to $B$ of $\Lambda$ is denoted by $\Pi_B(\Lambda)$ and it is referred to as the \emph{twin partition relative to $B$} of $\Lambda$. If $\A$ has no twins relative to B, that is, if $ \text{ for all } \, i,j \in \Lambda, \,i\neq j,\, D^1(i) \neq D^1(j), $
\noindent
then we say that $\A$ is \emph{twin-free relative to} B.  These definitions depend on the chosen natural basis (see \cite[Example 2]{YPMP}).

One of our main purposes is to study the derivations of Volterra evolution algebras. Given an (evolution) $\bK$-algebra $\A$, a  \emph{derivation} of $\A$ is a linear map $d : \A \rightarrow  \A$  such that 
\[d(u \cdot v)= d(u) \cdot  v + u \cdot d(v),\]
for all $u,v \in \A$. The space of all derivations of $\A$ is denoted by $\Der(\A)$.  In \cite[Section 3.2.6]{tian}, it was proved that, if $\A$ is a evolution $\bK$-algebra with a natural  basis $B=\{e_i \}_{i \in \Lambda}$  then a linear map $d$ such that $d(e_i)=\sum_{k\in \Lambda} d_{ki}e_k$ is a derivation of the evolution algebra $\A$ if, and only if, it satisfies the following conditions: 
\begin{eqnarray}
\omega_{jk}d_{ij}+ \omega_{ik}d_{ji}=0, &\,\,\, \text{for }i,j,k\in\Lambda\text{ such that }i\neq j,\label{eq:der1}\\
\sum_{k\in \Lambda} \omega_{ik}d_{kj}=2\omega_{ij}d_{ii}, &\,\,\, \text {for }i,j\in\Lambda,\label{eq:der2}
\end{eqnarray}
From now on, we identify the linear map $d$ with the matrix $(d_{ij})$ relative to the basis $B$.

\begin{remark}\label{decom} \rm
According  to \cite[Theorem 2.11]{boudi20} if $\A$ is an  evolution algebra with a natural basis $B=\{e_i\}_{i \in \Lambda}$ then we can write $B$ as a disjoint union of subsets as follows:
\begin{equation}\label{B_i}
B=B_0 \cup B_1 \cup \ldots \cup B_r,\end{equation}
where $\ann(\A)= \spa(B_0)$, $\rk(B_t)=1$ for $1\leq t \leq r$ and $\rk(\{u^2, v^2\})=2$ if $u\in B_t$ and $v \in B_s$ with $t\not= s$. Therefore,  if we define $\Lambda_t:=\{ k \in \Lambda:\, \, e_k \in B_t\}$, $\eqref{B_i}$ implies that $\Lambda$ can be also expressed as disjoint union of subsets:
\begin{equation}\label{Lambda_i}
    \Lambda=\Lambda_0 \cup \Lambda_1 \cup \ldots \cup \Lambda_r,
    \end{equation}
where $\rk(\{e_i^2, e_j^2\})=1$  if  $i \in \Lambda_t, j \in \Lambda_{s}$ and $t \neq s$ and $\rk(\{e_i^2,e_j^2\})=2$, if $i, j \in \Lambda_t$, for some $ t \in \{1,\dots, r \}$. So, in this last case,  \begin{equation}\label{alphas} 
e_i^2=\alpha_{ji} e_j^2\end{equation} for some $\alpha_{ji} \in \mathbb{K}^{\times}$.
\end{remark}

This observation leads us to the following definition.

\begin{defn}
\label{rem:alphas}
  \rm
In the same conditions of Remark \ref{decom}, the partitions $B=B_0 \cup B_1 \cup \ldots \cup B_r$ and $\Lambda=\Lambda_0 \cup \Lambda_1 \cup \ldots \cup \Lambda_r$ are called  \emph{a natural decomposition} of $B$ and \emph{a natural decomposition of  $\Lambda $ relative to $B$}, respectively.
 \end{defn}  
 
 \begin{defn} \label{lambdaj}\rm
 Let $\A$ an evolution algebra. We define $$\Lambda (j)=\{ k \in \Lambda\setminus\Lambda_0;\, \, e_k^2 \text{ and  }e_j^2\text{ are linearly dependent}\}.$$
 Moreover, we can write $e_k^2=\alpha_{jk} e_j^2$, for some $\alpha_{jk} \in \mathbb{K}^\times$ and $j,k \in \Lambda(j)$.
 \end{defn}

\begin{remark}\label{unicidaddes}

\rm
Under the conditions of Remark \ref{decom}, if $B'$ is another natural  basis  of $\A$ with a natural decomposition  $B'=B'_0 \cup B'_1 \cup \ldots \cup B'_s$  then  by \cite[Remark 2.14]{boudi20} we know that $r=s$ and that it is possible to reorder  $B'$ in such a way that $\spa (B_0)= \spa (B'_0)$ and $B'_t \subseteq \spa(B_0 \cup B_t)$. In addition, it is easy to check that $\vert B_t \vert = \vert B_t' \vert $ for every $t \in \{1, \ldots, r\}$.

 From now on, when we have natural decompositions of two natural bases $B$ and $B'$ of an evolution algebra $\A$ we suppose that both decompositions are written taking into account this reordering. 
\end{remark}

\section{Derivations of a non-degenerate evolution algebra}\label{dernon}
In this subsection, we will investigate when a linear operator of a non-degenerate evolution algebra is a derivation. In particular, we will study the derivations of an evolution algebra with $\dim(\A^2)=1$.

\smallskip

The following proposition improves \cite[Proposition 1]{YPMP} in the sense that it provides a condition necessary and sufficient  under which a linear operator $d: \A \to \A$ is a derivation of a non-degenerate evolution algebra $\A$.

\begin{prop}\label{PP}
Let $\A$ be a non-degenerate evolution algebra with  a natural basis $B=\{e_i \}_{i \in \Lambda}$, structure matrix $M_B=(\w_{ij})$ and  let  $d : \A \rightarrow  \A$ be a linear map, $d=(d_{ij})$. Then $d\in \Der(\A)$ if and only if $d$ satisfies the following conditions: 

\begin{enumerate}[label=(\roman*)]
\item If $i,j \in \Lambda, i \neq j$, and $i \sim_{t_B} j$ then $d_{ji}=-\frac{\w_{jk}}{\w_{ik}}d_{ij}$, for all $k \in D^1(i)$.\label{PP1}
\item If $i,j \in \Lambda$ and $i \not\sim_{t_B} j$ then $d_{ji}=d_{ij}=0.$ \label{PP2}
\item \label{PP3} For any $i\in \Lambda$
$$
\sum_{k\in D^1(i)} \w_{ik}d_{kj}=\left\{
\begin{array}{ll}
0,&\text{ if }j\notin D^1(i),\\[.2cm]
2\w_{ij}d_{ii},&\text{ if }j\in D^1(i).
\end{array}\right.
$$ 
\end{enumerate}
\end{prop}
\begin{proof}
If $d \in \Der(\A)$ then $d$ satisfies conditions \ref{PP1} to \ref{PP3} by \cite[Proposition 1]{YPMP}. Conversely, let  $d : \A \rightarrow  \A$ be a linear map satisfying conditions \ref{PP1} to \ref{PP3}. In order to prove that $d \in \Der(\A)$, it will be  necessary to  check that $d$ verifies (\ref{eq:der1}) and (\ref{eq:der2}). Let $i,j, k \in \Lambda$, $i\neq j$. If $i \sim_{t_B} j $, by \ref{PP1}, we have
$$\w_{jk}d_{ij}+ \w_{ik}d_{ji}=\w_{jk}d_{ij}+\w_{ik}\Big(-\frac{\w_{jk}}{ \w_{ik}}d_{ij}\Big)= 0, \quad \text{ for all }k\in D^1(i).
$$
Furthermore, if $k\not \in D^1(i)$ then $w_{jk}=w_{ik}=0$, which implies that $\w_{jk}d_{ij}+ \w_{ik}d_{ji}=0$. Otherwise if $i \not\sim_{t_B} j $, by \ref{PP2}, we have that $d_{ij}=d_{ji}=0$. Therefore $d$ satisfies  \eqref{eq:der1}.

Now note that for all $i,j \in \Lambda$ we have
$$\sum_{k=1} ^{n} \w_{ik}d_{kj}=\sum_{k \in D^1(i)} \w_{ik}d_{kj}=\left\{
\begin{array}{cl}
0,&\text{ if }j\notin D^1(i),\\[.2cm]
2\w_{ij}d_{ii},&\text{ if }j\in D^1(i).
\end{array}\right. =2\w_{ij}d_{ii}.$$
This proves that $d$ satisfies \eqref{eq:der2}.
\end{proof}

\begin{corollary}\label{blockmatrix}
Let $\A$ be a non-degenerate reducible evolution algebra with  $\A=I_1 \oplus \ldots \oplus I_t$ where $I_t$ is an ideal of $\A$ for every $t \in \{1,\ldots,m\}$. Then $d=(d_{ij})\in \Der{\A}$ is a block matrix. Moreover,  $d$ restricted to subspace $I_t$ (up to reordering) is a derivation over $I_t$ and its matrix is one of the blocks of $d=(d_{ij})$.
\end{corollary}

\begin{proof}
Let $B$ be a natural basis of $\A$. By \cite[Theorem 5.6]{YMV} we know that the ideals $I_1, \ldots,I_m$ can be taken as evolution ideals. Concretely, by  \cite[Theorem 5.6]{YMV} we get a partition of $B=B^{1}\cup \ldots \cup B^{m}$ such that $I_t=\spa(\{e_i \colon e_i\in B^{t}\})$ or equivalently the structure matrix relative to $B$ is a block diagonal matrix. Since $i \not\sim_{t_B} j$ for $e_i \in B^k$ and $e_j \in B^\ell$ with $k \neq \ell$ then by Proposition  \ref{PP} \ref{PP2} we have that  $d_{ij}=d_{ji}=0$. It is easy to check that for every $t\in \{1,\ldots,m\}$, $d$ restricted to $I_t$ is a derivation and moreover, if $d^t$ is the matrix of $d\vert_{I_t}$ relative to the natural basis $B^t$ then $$d=\diag{(d^1,\ldots,d^m)}.$$

\end{proof}

\begin{corollary}\label{corol:prop}
Let $\A$ be a non-degenerate evolution algebra  with  a natural basis $B=\{e_i \}_{i \in \Lambda}$ and structure matrix $M_B=(\w_{ij})$. If  $d=(d_{ij}) \in \Der(\A) $ then

\begin{enumerate}[label=(\roman*)]
\item \label{corol:prop1} If $i,j \in \Lambda$, $i \not =j$ and $d_{ij}\neq 0$ then $e_j^2=\alpha_{ij} e_i^2$, for some $\alpha_{ij} \in \mathbb{K}^\times$.
\item \label{corol:prop2} If $i,j \in \Lambda$ and $j \in D^1(i)$ then $$\displaystyle \sum_{k\in \T(j)} \w_{ik} d_{kj}=2\w_{ij}d_{ii}.$$ 

\end{enumerate}
\end{corollary}

\begin{proof} In order to prove \ref{corol:prop1} consider $i,j \in \Lambda$ such that $i \not =j$ and $d_{ij}\neq 0$. By \cite[Lemma 4]{YPMP} we have that  $i\sim_{t_B}j$. If $|D^{1}(i)|=1$ the proof is straightforward. In other case, by Proposition \ref{PP} \ref{PP1}, we have that
$$ d_{ji}=-\frac{\w_{jk}}{\w_{ik}}d_{ij}=-\frac{\w_{j\ell}}{\w_{i\ell}}d_{ij}, \text{ for all } k, \ell \in D^1(i).$$
Therefore $\frac{\w_{jk}}{\w_{ik}}=\frac{\w_{j\ell}}{\w_{i\ell}}.$ Fix $\ell \in D^1(i)$. Then
$$e_j^2=\sum_{k \in D^1(i)} \w_{jk}e_k=\sum_{k \in D^1(i)} \frac{\w_{j\ell}}{\w_{i\ell}}\w_{ik}e_k=\frac{\w_{j\ell}}{\w_{i\ell}}\sum_{k \in D^1(i)} \w_{ik}e_k=\frac{\w_{j\ell}}{\w_{i\ell}}e_i^2.$$
Taking $\alpha_{ij} :=\frac{\w_{j\ell}}{\w_{i\ell}}$, we have that $e_j^2=\alpha_{ij} e_i^2$, as required. For item \ref{corol:prop2}, we have that $d_{kj}=0$ for all $k \not\in \T(j)$ by Proposition \ref{PP} \ref{PP2}. Then, using Proposition \ref{PP} \ref{PP3}, we get 
$$2\w_{ij}d_{ii}=\sum_{k\in D^1(i)} \w_{ik}d_{kj}=\sum_{k\in \T(j)} \w_{ik}d_{kj}.$$
\end{proof}

\begin{corollary}\label{blockder}
Let $\A$ be a non-degenerate evolution algebra. Then $d=(d_{ij}) \in \Der(\A)$ can be written as a block matrix.
\end{corollary}

\begin{proof}
We consider a natural decomposition $B=B_0\cup \ldots \cup B_r$. We know by Corollary \ref{corol:prop} \ref{corol:prop1} that $d_{ij}=d_{ji}=0$ for every $e_i \in B_t$ and $e_j \in B_s$ with $t\neq s$.

\end{proof}

Observe that $\spa(B_t)$ for $t\neq 0$ is not an ideal in general, therefore $d$ restricted to subspace $\spa(B_t)$ is not necessarily a derivation.

\begin{prop}\label{prop34}
Let $\A$ be a non-degenerate evolution algebra with $\Der(\A)\neq 0$. Then $\A$ does not have an unique natural basis.
\end{prop}

\begin{proof} Let $B$ be a natural basis of $\A$  and let $d\in \Der(\A)$, with $d\neq 0$. By \cite[Lemma 1]{YPMP} there are $i,j \in \Lambda$ such that $i\neq j$ and $d_{ij} \neq 0$.  Therefore by Corollary \ref{corol:prop} \ref{corol:prop1} we have that $e_j^2=\alpha_{ij} e_i^2$, for some $\alpha_{ij} \in \mathbb{K}^\times$. Thus $\A$ has not the  Property \rm{(2LI)} and by \cite[Corollary 2.7]{boudi20} we get that $\A$ has not an unique natural basis.
\end{proof}

\begin{remark} \label{2li} \rm
Note that the Proposition \ref{prop34} is equivalent to say that if $\A$ has Property \rm{(2LI)} then  $\Der(\A)=0$. Since all perfect evolution algebras and evolution algebras which are twin-free both have the  Property \rm{(2LI)}, the Proposition \ref{prop34}  provides a generalization of the \cite[Theorem 2.1]{COT}, \cite[Theorem 4.1 item (1)]{Elduque/Labra/2019} and \cite[Theorem 1]{YPMP}. 

\end{remark}

However, the converse of Proposition \ref{prop34} is not true as shown the following example.
\begin{exa} \rm
Let $\A$ a non-degenerate two dimensional evolution algebra with product $e_1^2=e_2^2=e_1+e_2$. As $\A$ does not have Property \rm{(2LI)}, then does not have an unique natural basis. However, it is easy to check that $\Der(\A)=0$.
\end{exa}

\begin{prop} \label{dii}
Let $\bK$ be an arbitrary field and let $\A$ be a non-degenerate evolution $\bK$-algebra with $\dim(\A^2)=1$. Consider $B=\{e_i\}_{i \in \Lambda}$ a natural basis  and $M_B=(\omega_{ij})$ the structure matrix. For $i\in \Lambda$, let $\alpha_{1i}$ be no null scalars such that $e_i^2=\alpha_{1i} e_1^2$. Suppose that $e_1^2e_1^2\neq 0$. Then $d\in \Der (\A)$ if and only if it verifies the following conditions:

\begin{enumerate}[label=(\roman*)]
       \item $d_{ii}=0$ for any $i \in \Lambda$. \vspace{0.2cm}
    \item $d_{ij}= - \dfrac{\alpha_{1i}}{\alpha_{1j}} d_{ji}$ for any $i,j \in \Lambda$, $i \neq j$.\vspace{0.1cm}
    \item $\sum_{j\in \Lambda}\omega_{1j}d_{jk}=0$ for $k \in \Lambda$.
\end{enumerate}
 
\end{prop}

\begin{proof}
Firstly, observe that  $\A^2=\bbK e_1^2$. We write $e_1^2 e_1^2= \gamma e_1^2$ for some $\gamma \in \bbK^{\times}$. Let $d \in \Der(\A)$. We get that $d(e_i^2)=2e_id(e_i)=2d_{ii}e_i^2$ for any $i \in \Lambda$. On the other hand, we have that $d(e_1^2 e_1^2)=2e_1^2d(e_1^2) $, so $\gamma d(e_1^2)=2e_1^2 d(e_1^2)$. Therefore, if  ${\rm char}(\bbK) =2$, then $d(e_1^2)=0$. If  ${\rm char}(\bbK) \not =2$, we get that $\gamma d_{11}e_1^2 =2d_{11}e_1^2 e_1^2$. This implies that  $\gamma d_{11}e_1^2 =2d_{11} \gamma e_1^2$. So, as $\gamma \neq 0$ and $\A$ is non-degenerate, then $d_{11}=0$. We conclude that $d(e_1^2)=0$ in any case and therefore $d(e_i^2)=0$ for any $i \in \Lambda$. Then $d_{ii}=0$ because $d(e_i^2)=2d_{ii}\alpha_{1i} e_1^2$. 
 Let $i \neq j$ with $i,j \in \Lambda$, then $d(e_i e_j)=0$, which implies that $d_{ij}e_j^2+d_{ji}e_i^2=0$. So $e_1^2(d_{ij}\alpha_{1j}+d_{ji}\alpha_{1i})=0$ and by non-degeneracy of $\A$ $d_{ij}=-\frac{\alpha_{1i}}{\alpha_{1j}}d_{ji}$. Finally, as $d(e_1^2)=2d(e_1)e_1$ by \eqref{eq:der2} we have that  $\sum_{j}\omega_{1j}d_{jk}=2d_{11}\omega_{1k}$ for every $k \in \Lambda$. The converse is straightforward. 
\end{proof}

\begin{corollary}\label{dimen1}
Let $\A$ be a non-degenerate evolution algebra such that $\dim(\A^2)=1$ and  natural basis $B=\{e_i\}_{i \in \Lambda}$. Suppose that $e_i^2=e_1^2$  for any $i \in \Lambda$ and $e_1^2e_1^2 \neq 0$.  Then if $d\in \Der (\A)$, the matrix of $d$ relative to $B$ is skew-symmetric, up to reordering.
\end{corollary}

\begin{remark} \rm
The converse of Corollary \ref{dimen1} is not true in general. Indeed, if we consider the $3$-dimensional evolution algebra $\A$ with basis $\{e_i\}$ and product $e_i^2=e_1$ for any $i=\{1,2,3\}$ then $d \in \Der(\A)$ if and only if $d_{ii}=0$ for any $i$, $d_{12}=d_{21}=d_{13}=d_{31}=0$ and $d_{32}=-d_{23}$.
\end{remark}

The condition $e_1^2 e_1^2\neq 0$ can not be eliminated of the Proposition \ref{dii} as the following remark shows.

\begin{remark} \rm
Let $\A$ be a non-degenerate evolution algebra  with $\dim(\A^2)=1$, natural basis $B=\{e_i \}_{i \in \Lambda}$ and structure matrix $M_B=(\w_{ij})$. We can write $e_i^2=\alpha_{1i}e_1^2$ for every $i \in \Lambda$. Suppose that $e_1^2 e_1^2=0$. Since $e_1^2\neq 0$ there exists $k$ such that $\omega_{1k} \neq 0$. Now, we will find a derivation such that $d_{ii} \neq 0$ for every $i \in \Lambda$. Indeed, it is enough to consider the derivation $d \in \Der(\A)$ defined by 
$$d_{ij}=\left\{ \begin{array}{lcl}
             1, &   if  &  i =j,\\[0.1cm]
              -\dfrac{\alpha_{1i}\omega_{1i}}{\alpha_{1k}\omega_{1k}}, &  if &  j=k \, \,\,  {\rm and}  \, \, \,  i \in \Lambda \setminus \{k\},\\[0.1cm]
             0, &  if  & i \neq j \, \,\,{\rm and} \, \,\, i,j \neq k,\\[0.1cm]
              \dfrac{\omega_{1j}}{\omega_{1k}}, &  if &  i \neq j \, \,\,  {\rm and}  \, \, \,  i =k.\\
             \end{array}
   \right.
   $$

\end{remark}

\begin{prop}\label{diago1}
Let $\A$ be an evolution algebra with  $\{e_i\}_{i\in \Lambda}$  natural basis and $d=(d_{ij})\in \Der(\A)$. Let $\{e_i^2\}_{i=1}^\ell$ be a basis of $\A^2$ and  $e_{j}^2=\sum_{k=1}^\ell \beta_{kj} e_k^2$. If $\beta_{kj} \neq 0$ for certain $k \in \Lambda$ then $d_{jj}=d_{kk}$ for any $j \in \Lambda$.
\end{prop}

\begin{proof}
Let $d \in \Der(\A)$. If $j\in \{1,\ldots,\ell\}$ then the statement is trivially true. We study now if $j\notin \{1,\ldots,\ell\}$. Applying $d$ in the equality $e_{j}^2=\sum_{k=1}^\ell \beta_{kj} e_k^2$ we get that $e_{j}^2d_{jj}=\sum_{k=1}^\ell \beta_{kj} e_k^2 d_{kk}$. So $d_{jj}\sum_{k=1}^\ell \beta_{kj} e_k^2=\sum_{k=1}^\ell \beta_{kj} e_k^2 d_{kk} $. Then $\sum_{k=1}^\ell \beta_{kj}(d_{jj}-d_{kk})e_{k}^2=0$.   Since $\{e_i^2\}_{i=1}^\ell$ is a basis of $\A^2$ then $\beta_{kj}(d_{jj}-d_{kk})=0$ for every $k \in \{1,\ldots,\ell\}$. If  there exists $k \in \Lambda$ such that $\beta_{kj} \neq 0$ then $d_{jj}=d_{kk}$.

\end{proof}

\begin{theorem} \label{diagonalceros}

Let $\A$ be a non-degenerate evolution algebra with $\{e_i\}_{i\in \Lambda}$ natural basis and $d=(d_{ij})\in \Der(\A)$. If $\{e_i^2\}_{i=1}^\ell$ is a basis of $\A^2$ with $e_i^2e_i^2 \neq 0$ for any $i \in \Gamma_1:=\{1,\ldots, \ell\}$ then $d_{jj}=0$ for any $j\in \Lambda$.

\end{theorem}

\begin{proof}
Let $d \in  \Der(\A)$. First, we can write $e_i^2 e_i^2= \sum_{k=1}^\ell \lambda_{ik} e_k^2$ for any $i \in \Lambda$. Let $j \in \Lambda$ then if we apply the derivation $d$ in both members of  $e_j^2 e_j^2= \sum_{k=1}^\ell \lambda_{jk} e_k^2$ we get $2 e_j^2 d(e_j^2)= 2 \sum_{k=1}^\ell \lambda_{ji}e_k d(e_k)$, so $2e_j^2e_j^2 d_{jj} = \sum_{k=1}^\ell\lambda_{jk}e_k^2 d_{kk}  $. Therefore  $2 \sum_{k=1}^\ell \lambda_{jk} e_k^2 d_{jj} = \sum_{k=1}^\ell\lambda_{jk}e_k^2 d_{kk}$, which implies $\sum_{k=1}^\ell \lambda_{jk}(2d_{jj}-d_{kk})e_k^2=0$. Since $\{e_k^2\}_{k=1}^\ell$ is a basis of $\A^2$ then  $\lambda_{jk}(2d_{jj}-d_{kk})=0$ for any $k \in \Gamma_1$. As $e_j^2e_j^2 \neq 0$ and $\A$ is non-degenerate there exists some $j_1 \in \Gamma_1$ such that $\lambda_{jj_1}\neq 0$ and so $2d_{jj}-d_{j_1j_1}=0$. Firstly, we consider the set $R=\{j \in \Gamma_1 \, \colon \, 2d_{jj}-d_{jj}=0\}$, then $d_{jj}=0$ for any $j \in R$. Secondly, let $j_0 \in \Gamma_1$ such that $j_0 \notin R$. We can write the following chain of equalities: 
$$\begin{array}{rcl}
2d_{j_0j_0}&=&d_{j_1j_1},\\
2d_{j_1j_1}&=&d_{j_2j_2}, \\
& \vdots &\\
2 d_{j_{s-1}j_{s-1}}&=&d_{j_sj_s},
\end{array}$$
with $j_1,\ldots,j_s \in \Gamma_1$. Moreover either $j_s \in R$ or $j_s \notin R$ but as $\ell$ is finite, $j_s=j_{m}$ for certain $j_m \in \{j_0,j_1,\ldots,j_{s-2}\}$. In the first case we get that  $d_{j_0j_0}=d_{j_1j_1}=\ldots=d_{j_s j_s}=0$. In the second case, if we  write $s=m+t$ for $t > 1$ then it is easy to check that $2^td_{j_mj_m}=d_{j_{m+t}j_{m+t}}$ i.e., $2^td_{j_mj_m}=d_{j_{m}j_{m}}$. Then $d_{j_mj_m}=d_{j_0j_0}=\cdots=d_{j_{s-1}j_{s-1}}=0$. Therefore if $j \in \Gamma_1$ we have proved that $d_{jj}=0$. Let $j \notin \Gamma_1$. Now, we know that $2d_{jj}-d_{j_1j_1}=0$ for certain $j_1 \in \Gamma_1$ so $d_{jj}=0$.

\end{proof}

\begin{remark} \rm In terms of matrices, for every $i \in \Lambda$ we can compute the product $e_i^2 e_i^2$ as $(M_B \circ M_B)\cdot M_B \cdot (e_1\ldots e_n)^t$ where $\circ$ is the Hadamard product (element-wise multiplication). 

\end{remark}

The converse of Theorem \ref{diagonalceros} is not true in general as shown the following example. 

\begin{exa} \rm
Let $\A$ an evolution algebra with $B=\{e_i\}_{i\in \Lambda}$ ($\Lambda=\{1,\ldots,5\}$) natural basis and multiplication table $e_1^2=e_1+e_2+e_4+e_5$, $e_2^2=e_1+e_2$, $e_3^2=e_4+e_5$, $e_4^2=-e_5^2=e_3$ and  $d=(d_{ij})\in \Der(\A)$. Since $\{e_1^2, e_2^2,e_4^2\}$ is a basis of $\A^2$ and $e_i^2e_i^2 \neq 0$ for every $i\in \{1,2,4\}$ then $d_{ii}=0$ for every $i\in  \Lambda$. But if we consider $\{e_1^2, e_3^2, e_4^2\}$ a basis of $\A^2$ then $e_3^2e_3^2=0$ and clearly $d_{ii}=0$ for every $i \in \Lambda$.
\end{exa}

\section{Derivations of Volterra Evolution Algebras} \label{volterrader}

\begin{lemma}
Let $\A$ be a non-degenerate Volterra evolution algebra with natural basis $B=(\w_{ij})$ and let $i,j \in \Lambda$ such that $\T(i)=\{i,j\}$ and there exists $\ell \in D^1(i)$ with $\w_{\ell i}^3 \neq \w_{j\ell}^3$. Suppose that  $\ell\nsim_{t_B} k$  for every $k \in D^1(i)\setminus \{ \ell\}$. Then  $d_{ij}=d_{ji}=d_{ii}=d_{jj}=d_{\ell\ell}=0$ for any $d=(d_{ij}) \in \Der(\A)$.
\end{lemma}

\begin{proof}
Let $d=(d_{ij}) \in \Der(\A)$. Since $\displaystyle\sum_{k\in D^1(i)}\w_{ik}d_{k\ell}=2\w_{i\ell}d_{ii}$ and $\ell\nsim_{t_B} k$ we get that $d_{\ell\ell}=2d_{ii}$ by Proposition \ref{PP} \ref{PP2}.
Likewise, $d_{\ell\ell}=2d_{jj}$. On the other hand by Proposition \ref{PP} \ref{PP1}, $d_{ji}=-\frac{\w_{j\ell}}{\w_{i\ell}}d_{ij}$.
Moreover, since $\displaystyle\sum_{k\in D^1(\ell)}\w_{\ell k}d_{ki}=2\w_{\ell i}d_{\ell}$ and $\T(i)=\{i,j\}$ then $\w_{\ell i}d_{ii}+\w_{\ell j}d_{ji}=2\w_{\ell i}d_{\ell\ell}$. Similarly we have $\w_{\ell i}d_{ij}+\w_{\ell j}d_{jj}=2\w_{\ell j}d_{\ell\ell}$. So, we get the following homogeneous system of linear equations:
$$-3\w_{\ell i}\w_{i\ell}d_{ii}- \w_{\ell j}\w_{j\ell}d_{ij}=0, $$
$$-3\w_{\ell j}d_{ii}+\w_{\ell i}d_{ij}=0.$$
This system will have the trivial solution if and only if $\w_{\ell i}^3 \neq \w_{j\ell}^3$.
\end{proof}

\begin{prop}\label{pro:dervolt}
Let $\A$ a non-degenerate Volterra evolution algebra with  a natural basis $B=\{e_i\}_{i \in \Lambda }$ and structure matrix $M_B=(\w_{ij})$. Consider  a natural decomposition $\Lambda=\Lambda_1\cup \ldots \cup \Lambda_r$ relative to $B$ and $\alpha_{ij} \in \bbK^{\times}$ such that $e_j^2=\alpha_{ij}e_i^2$ for $i,j\in \Lambda(i)$. Then $d=(d_{ij}) \in \Der(\A)$ if and only if $d$ satisfies the  following conditions:
\begin{enumerate}[label=(\roman*)]
    \item If $i, j \in \Lambda$, $i \neq j$ and  $\{ i,j \}\not\subseteq \Lambda_{t}$ for any ${t} \in \{1,\dots, r\} $ then   
    $d_{ij}=d_{ji}=0$. \label{pro:dervolt1}
    \item If $ i,j \in \Lambda$, $i \neq j$ and  $\{ i,j \}\subseteq \Lambda_{t}$ for some $t \in \{1,\dots, r\} $ then  $d_{ij}=-\alpha_{ji}d_{ji}$.  \label{pro:dervolt2}
    \item If $ i, j \in \Lambda$ and $ i \in D^1(j)$ then  $\displaystyle 2d_{ii}= \sum_{k \in \Lambda(j)} \alpha_{jk}
    d_{kj}$.  \label{pro:dervolt3}
\end{enumerate}
\end{prop}

\begin{proof}
If $i \neq j$ and $\{i,j\} \not\subseteq \Lambda_{t}$ for any ${t} \in \{1,\dots, r\} $, then $e_i^2$ and $e_j^2$ are linearly independent. By Corollary \ref{corol:prop} \ref{corol:prop1} we have $d_{ij}=d_{ji}=0$, which proves item \ref{pro:dervolt1}. Now, note that if $i \neq j$ and $\{i,j\} \subseteq \Lambda_{t}$ for some ${t} \in \{1,\dots, r\}$, then  $e_i^2=\alpha_{ji}e_j^2$ and $w_{ik}=\alpha_{ji} w_{jk}$ for all $k \in D^1(j)$. By Proposition \ref{PP} \ref{PP1} we have
$$d_{ij}=-\frac{w_{ik}}{w_{jk}}d_{ji}=-\alpha_{ji}d_{ji},$$
which proves item \ref{pro:dervolt2}. Now, let $i, j \in \Lambda$.   By item \ref{pro:dervolt1}, if $k \not\in \Lambda(j)$  then $d_{kj}=0$. We have 
\begin{eqnarray} \label{eq:2diivolterra}
\sum_{k\in \Lambda} w_{ik}d_{kj}&=&\sum_{k\in \Lambda(j)}  w_{ik}d_{kj}=\sum_{k\in  \Lambda(j)} -w_{ki}d_{kj}=\sum_{k\in  \Lambda(j)}-\alpha_{jk} w_{ji}d_{kj}\\
 &=&-w_{ji}\sum_{k\in  \Lambda(j)} \alpha_{jk} d_{kj}=w_{ij}\sum_{k\in  \Lambda(j)} \alpha_{jk} d_{kj} \nonumber
\end{eqnarray}
for any $i,j\in \Lambda$. On the other hand, using Equation \eqref{eq:der2} we get  $2 w_{ij}d_{ii}=\sum_{k=1}^{n} w_{ik}d_{kj}$ then $2 w_{ij}d_{ii}=w_{ij}\sum_{k\in  \Lambda(j)} \alpha_{jk} d_{kj}$. Now, if $i \in D^1(j)$ therefore $2d_{ii}=\sum_{k\in  \Lambda(j)} \alpha_{jk} d_{kj}$, which proves item \ref{pro:dervolt3}. Conversely, let $d=(d_{ij})$ satisfying conditions \ref{pro:dervolt1}-\ref{pro:dervolt3}. We will prove that $d$ satisfies conditions \ref{PP1}-\ref{PP3} of Proposition \ref{PP}.

Let $i,j \in \Lambda$, $i\neq j$ and $i \sim_t j$. 

\noindent\textbf{Case 1.} If $\{ i,j \} \not\subseteq \Lambda_{t}$ for any $t\in \{1,\dots,r\}$. Then by item \ref{pro:dervolt1} $d_{ij}=d_{ji}=0$.

\noindent\textbf{Case 2.} If $\{ i,j \} \subseteq \Lambda_{t}$ for some $t \in \{1,\dots,r\}$. Then by item \ref{pro:dervolt2} $d_{ij}=-\alpha_{ji}d_{ji}$. Note that $\alpha_{ji}=\frac{w_{ik}}{w_{jk}}$ for all $k \in D^1(i)$. 

Therefore, for both cases, we have that $d_{ji}=-\frac{w_{jk}}{w_{ik}} d_{ij}$ for all $k \in D^1(i)$. 

Let $i,j \in \Lambda$, $i\neq j$ and $i \not\sim_t j$. Then $e_i^2$ and $e_j^2$ are linearly independent and by Corollary \ref{corol:prop} \ref{corol:prop1} we get that $d_{ij}=d_{ji}=0.$ 

Let $i \in \Lambda$, by Equation \eqref{eq:2diivolterra} and item \ref{pro:dervolt3} we obtain that
$$
\sum_{k\in \Lambda} \w_{ik}d_{kj}=w_{ij}\sum_{k\in  \Lambda(j)} \alpha_{jk} d_{kj}=
\left\{
\begin{array}{cl}
0&\text{ if }j\notin D^1(i)\\[.2cm]
2\w_{ij}d_{ii}&\text{ if }j\in D^1(i).
\end{array}\right.
$$ 
\end{proof}

\begin{corollary}
Let $\A$ be a non-degenerate Volterra evolution algebra and $d=(d_{ij})\in \Der(\A)$. If $i,\ell \in D^1(j)$ for some $j \in \Lambda$ then $d_{ii}= d_{\ell \ell}$.

\end{corollary}
\begin{proof} The result is a direct consequence of Proposition \ref{pro:dervolt} \ref{pro:dervolt3}.
\end{proof}

\begin{theorem}\label{teo:dervolt}
Let $\A$  be a non-degenerate Volterra evolution algebra with a natural basis $B=\{e_i\}_{i \in \Lambda }$ and structure matrix $M_B=(\w_{ij})$. Let  $\Lambda=\Lambda_1 \cup \ldots \cup \Lambda_r$ be a natural decomposition and $\alpha_{ij} \in \bbK^{\times}$ such that $e_j^2=\alpha_{ij}e_i^2$ for $i,j\in \Lambda(j)$. Define $\nu_t:=|\Lambda_t|$ for all $t\in \{1,\dots, r\}$. If $e_i^2e_i^2\neq 0$ for all $i \in \Lambda$ then there exists a Volterra evolution algebra $\A'$ with a natural basis $B'$ such that $\Der(\A)=\Der(\A')$. Moreover, the structure matrix $M_{B'}$ is the following block diagonal matrix $$\diag({H_1,\ldots,H_c}),$$ 
 where $r=2c+q$, $q\in \{0,1\}$ and $H_t$ is given by:
  \begin{enumerate}[label=(\roman*)]
    \item  \label{teo:dervolt:even} If $q=0$ then  for $t \in \{1,\dots, c\}$ we have that 
    \begin{equation}\label{eq:even}H_t=\left(
    \begin{matrix} 
    0 & F_t \\
    -F_t^T & 0
    \end{matrix}\right),\end{equation}
    where  $F_t \in M_{\nu_{t},\nu_{t+1}} (\bbK)$  is a matrix without null entries.
    \item  \label{teo:dervolt:odd} If $q=1$  then for  $t \in \{1,\ldots,c-1\}$ the matrix $H_t$ is of the form \eqref{eq:even} and \begin{equation}\label{eq:odd}H_c=\left(
    \begin{matrix} 
    0 & F_c & F_{c+1} \\
    -F_c^T & 0 & 0 \\
    -F_{c+1}^T & 0 & 0
    \end{matrix}\right)
   \end{equation} 
     $F_{c} \in M_{\nu_{r-2},\nu_{r-1}} (\bbK)$ and  $F_{c+1} \in M_{\nu_{r-2},\nu_{r}} (\bbK)$are  matrices
    without null entries.
    \end{enumerate}
\end{theorem}
\begin{proof} We define $s_0=0$ and $s_{h}=\sum_{k=1}^{h} \nu_{k}$ with $h\in \{1,\ldots,r\}$. Note that we can reorder $\Lambda$ such that $\Lambda_h=\{s_{h-1}+1,\ldots,s_{h}\}$ for $h\in \{1,\ldots,r\}$. First, we assume that $r$ is even.  We are going to construct an evolution algebra  $\A'$ with  natural basis  $B'=\{f_j\}_{j \in \Lambda }$. To describe the product in $\A'$, we consider  $t \in \{1,\ldots,c\}$ and denote by $p=s_{2t-2}+1$ and $q=s_{2t-1}+1$. Now, with this notation we define
\begin{alignat}{2}
f_{p}^2&:=\sum_{k \in \Lambda_{2t}} \alpha_{qk} f_k, \label{eq:fi1} \\ f_{j}^2&:=\alpha_{pj}f_{p}^2, \text{ for all } j \in \Lambda_{2t-1}\setminus\{ p\},\label{eq:fil} \\ 
f_{q}^2&:=-\sum_{k \in \Lambda_{2t-1}} \alpha_{pk} f_k, \label{eq:fj1}\\  f_{j}^2&:=\alpha_{qj}f_{q}^2, \text{ for all j}  \in \Lambda_{2t}\setminus\{ q\}.\label{eq:fjt}
\end{alignat}

Note that $D^1(\Lambda_{2t-1})=\Lambda_{2t}$ and $D^1(\Lambda_{2t})=\Lambda_{2t-1}$ for $t \in \{1,\dots, c\}$. 
So, we get that the structure matrix of $\A'$ relative to that basis is given by the diagonal matrix $\diag{(H_1,\ldots,H_c)}$ with $$H_{t}=\left(
    \begin{matrix} 
    0 & F_t \\
    -F_t^T & 0
    \end{matrix}\right)$$ and
\begin{equation}\label{F_i}F_t:=
\begin{pmatrix}
1& \alpha_{qq+1}&\dots &\alpha_{qs_{2t}} \\
\alpha_{pp+1} & \alpha_{pp+1}\alpha_{qq+1}&\dots &\alpha_{pp+1}\alpha_{qs_{2t}} \\
\vdots & \vdots &\ddots &\vdots \\
\alpha_{ps_{2t-1}} & \alpha_{ps_{2t-1}}\alpha_{qq+1}&\dots &\alpha_{ps_{2t-1}}\alpha_{qs_{2t}} 
\end{pmatrix}\end{equation}
for every $t\in \{1,\ldots,c\}$. Consequently $\A'$ is a Volterra evolution algebra. Observe that for all $i,j \in \Lambda(j)$ we have
$e_i^2=\alpha_{ji}e_j^2 \text{ and } f_i^2=\alpha_{ji}f_j^2$. As $e_i^2e_i^2\neq 0$ for all $i \in \Lambda$, then by Theorem \ref{diagonalceros} we obtain that $d_{ii}=0$ for all $i \in \Lambda$. Thus, by Proposition \ref{pro:dervolt} follows that $\Der(\A)=\Der(\A')$, as required. 

Now, we assume that $r$ is odd. In this case we consider an evolution algebra $\A'$ with natural basis $B'=\{f_j\}_{j \in \Lambda }$. To define the product in $\A'$  we consider $t \in \{1,\dots, c-1\}$ and we define
$f_j$ as in Equations (\ref{eq:fi1}-\ref{eq:fjt}) for $j \in \Lambda_{2t-1}\cup \Lambda_{2t}$. We denote by
$p=s_{r-3}+1$, $q=s_{r-2}+1$ and $\ell=s_{r-1}+1$ and we define
\begin{alignat*}{2}
f_{p}^2&:=\sum_{k \in \Lambda_{r-1}} \alpha_{qk} f_k+\sum_{k \in \Lambda_{r}} \alpha_{\ell k} f_k, \\ f_{j}^2&:=\alpha_{pj}f_{p}^2, \text{ for all } j \in \Lambda_{r-2}\setminus\{p\}, \\ 
f_{q}^2&:=-\sum_{k \in \Lambda_{r-2}} \alpha_{pk} f_k, \\  f_{j}^2&:=\alpha_{qj}f_{q}^2, \text{ for all } j \in \Lambda_{r-1}\setminus\{q\}, \\
f_{\ell}^2&:=-\sum_{k \in \Lambda_{r-2}} \alpha_{pk} f_k; \\  f_{j}^2&:=\alpha_{\ell j}f_{\ell}^2, \text{ for all } j \in \Lambda_{r}\setminus\{ \ell\}. 
\end{alignat*}

We have that $D^1(\Lambda_{2t-1})=\Lambda_{2t}$ and $D^1(\Lambda_{2t})=\Lambda_{2t-1}$ for $t \in \{1,\ldots, c-1\}$, $D^1(\Lambda_{r-2})=\Lambda_{r-1}\cup\Lambda_{r}$ and $D^1(\Lambda_{r-1})=D^1(\Lambda_{r})=\Lambda_{r-2}$. Furthermore, as in the previous case, we can reorder $B'$ such that the structure matrix of $\A'$ relative to that basis is a block diagonal matrix $\diag{(H_1,\ldots,H_c)}$ where $H_t$ is a matrix of the form \eqref{eq:even},  $F_t$ of the form \eqref{F_i} for $t\in \{1, \ldots, c-1\}$, $H_c$ is of the form \eqref{eq:odd} and 
$$F_c=
\begin{pmatrix}
1 & \alpha_{qq+1}&\dots &\alpha_{qs_{r-1}} \\
\alpha_{pp+1} & \alpha_{pp+1}\alpha_{qq+1}&\dots &\alpha_{pp+1}\alpha_{qs_{r-1}} \\
\vdots & \vdots &\ddots &\vdots \\
\alpha_{ps_{r-2}} & \alpha_{ps_{r-2}}\alpha_{qq+1}&\dots &\alpha_{ps_{r-2}}\alpha_{qs_{r-1}} 
\end{pmatrix},$$

$$F_{c+1}=
\begin{pmatrix}
1 & \alpha_{\ell \ell +1}&\dots &\alpha_{\ell s_{r}} \\
\alpha_{pp+1} & \alpha_{pp+1}\alpha_{\ell \ell +1}&\dots &\alpha_{pp+1}\alpha_{\ell s_{r}} \\
\vdots & \vdots &\ddots &\vdots \\
\alpha_{ps_{r-2}}& \alpha_{ps_{r-2}}\alpha_{\ell \ell +1}&\dots &\alpha_{ps_{r-2}}\alpha_{\ell s_{r}} 
\end{pmatrix}.$$

Analogously to other case, if $i,j \in \Lambda(j)$, we have
$e_i^2=\alpha_{ji}e_j^2 \text{ and } f_i^2=\alpha_{ji}f_j^2$ and $d_{ii}=0$ for every $i \in \Lambda$. Consequently, $\Der(\A)=\Der(\A')$.
\end{proof}

\begin{remark} \rm
Thanks to the Corollary \ref{blockmatrix} since $\A'$ is a non-degenerate reducible evolution algebra, to compute the set of derivations of $\A$ is equivalent to compute the set of derivations over the evolution ideals $I_{i}=\spa(\{e_j \colon j \in \Lambda_i \cup \Lambda_{i+s}\})$ for any $i\in \{1,\ldots,s-1\}$, $I_{s}=\spa(\{e_j \colon j \in \Lambda_s \cup \Lambda_{2s}\})$ if $r=2s$ and $I_{s}=\spa(\{e_j \colon j \in \Lambda_s \cup \Lambda_{2s} \cup \Lambda_{2s+1}\})$ if $r=2s+1$. So, we can reduce the dimension of evolution algebras whose space of derivations will be studied.  

\end{remark}

\begin{prop}\label{prop:sumalpha3=0}
Let $\A$ be  a non-degenerate  Volterra evolution algebra with  a natural basis $B=\{e_i\}_{i \in \Lambda }$ and  $\Lambda=\Lambda_1\cup \ldots \cup \Lambda_r$ a natural decomposition of $\Lambda$ relative to $B$. Let $\alpha_{ij} \in \bbK^{\times}$ such that $e_j^2=\alpha_{ij}e_i^2$ for $i,j\in \Lambda(j)$.  Suppose that there exists
$i \in \Lambda$ such that for any $j\in D^1(i)$  it is verifies that $ \sum_{k \in \Lambda(j)}  \alpha_{jk}^3=0$ then $e_i^2e_i^2=0$.
\end{prop}

\begin{proof}
Let $M_B=(\w_{ij})$ be the  structure matrix  of $\A$ relative to $B$ and $i \in \Lambda$ such that for any $j\in D^1(i)$ we have that $ \sum_{k \in \Lambda(j)}  \alpha_{jk}^3=0$. Note that there exists $\{\ell_1,\ldots,\ell_t\} \subseteq D^1(i)$ such that $D^1(i)=\Lambda({\ell_1})\cup  \ldots \cup \Lambda({\ell_t})$ with $\Lambda(\ell_h)\cap \Lambda(\ell_g)=\emptyset$ for $h\neq g$. Then 

\[\begin{array}{rclcl}
\displaystyle e_i^2e_i^2
&=& \displaystyle \sum_{k \in \Lambda({\ell_1})} w_{ik}^2 e_k^2+\ldots + \sum_{k \in \Lambda({\ell_t})} w_{ik}^2 e_k^2\\[0.3cm]
&=&\displaystyle \sum_{k \in \Lambda({\ell_1})} w_{ki}^2\alpha_{\ell_1k} e_{\ell_1}^2+\ldots+
\sum_{k \in \Lambda({\ell_t})} w_{ki}^2\alpha_{\ell_tk} e_{\ell_t}^2 \\[0.3cm]
&=& \displaystyle w_{\ell_1i}^2e_{\ell_1}^2\sum_{k \in \Lambda({\ell_1})} \alpha_{\ell_1k}^3+\ldots+
w_{\ell_ti}^2 e_{\ell_t}^2\sum_{k \in \Lambda({\ell_t})} \alpha_{\ell_tk}^3.
\end{array}\]
Since $\ell_{h}\in D^1(i)$ for $h \in \{1,\ldots, t\}$ then by hypothesis $\displaystyle \sum_{k \in \Lambda({\ell_{1}})} \alpha_{\ell_1k}^3=\dots=\sum_{k \in \Lambda({\ell_{t}})} \alpha_{\ell_tk}^3=0.$ Therefore $e_i^2e_i^2=0$.
 \end{proof}

\begin{remark} \rm 
Theorem \ref{diagonalceros} presents conditions for a non-degenerate evolution algebra to have only derivations with zero diagonal. If $\A$ is degenerate, then there always exists $d \in \Der(\A)$ with non-zero diagonal entries. Indeed, let $\A$ be a degenerate evolution algebra with a natural basis $B=\{e_i\}_{i \in \Lambda }$ and $\ell \in \Lambda$ such that $e_{\ell}^2=0$. Consider the linear operator $d=(d_{ij})$ defined by 
$$d_{ij}=\left\{ \begin{array}{ll}
             1, &   \text{ if }   i=j=\ell,\\
              0, &  \text{ if }i \neq \ell \text{ or } j \neq \ell.
             \end{array}
   \right.
   $$
\noindent Then the Equations \eqref{eq:der1} and \eqref{eq:der2} shows that $d \in \Der(\A)$.
The next two results show conditions for a non-degenerate Volterra evolution algebra to have derivations with non-zero diagonal entries
\end{remark}

\begin{remark} \rm
Let $\A$ be  a Volterra evolution algebra with a natural basis $B=\{ e_i\}_{i \in \Lambda}$ and  let $\Lambda=\Lambda_1 \cup  \ldots\cup \Lambda_r$ be a natural decomposition of $\Lambda$ relative to $B$.
Note that if  there exists a path from $i$ to $j$ then $\delta (i,\ell)=\delta(i,j)$ for all $\ell \in \Lambda(j)$. Therefore, $\delta (i,\Lambda(j))=\delta (i,j)$.  Analogously, we have that $\delta (\Lambda(i),\Lambda(j))=\delta (i,j)$.
\end{remark}

\begin{prop} \label{prop:e^4diineq0}
Let $\A$  be a non-degenerate Volterra evolution algebra with a natural basis $B=\{e_i\}_{i \in \Lambda }$  and  $\Lambda=\Lambda_1\cup \ldots \cup \Lambda_r$ a natural decomposition of $\Lambda$ relative to $B$. Assume that the associated graph $E^B_{\A}$ does not have odd length cycles. Moreover, suppose that there is $i \in \Lambda$ such that  $\displaystyle \sum_{k \in \Lambda(j)}  \alpha_{j k}^3=0$ for all  $j \in \Lambda$ with $\delta(i,j)$ even. Then there exists $d=(d_{ij}) \in \Der(\A)$ verifying $d_{kk}\neq 0$ for any $k \in D(i)$.
\end{prop}

\begin{proof} Firstly, we observe that if $i \in \Lambda$ then either $\Lambda_h \subseteq D(i)$ or $\Lambda_h \cap D(i) =\emptyset$ for every $h \in \{1,\ldots,r\}$. Now, we can suppose without loss of generality that there exists a reordering of the natural decomposition in such a way that
$\Lambda=\Lambda_1 \cup  \ldots \cup \Lambda_t\cup \Lambda_{t+1} \cup \ldots \cup \Lambda_{r}$ where  $i \in \Lambda_1$,  $\Lambda_h \subseteq D(i)$ for all $h \in \{1,\dots, t\}$ and $\Lambda_h \cap D(i)=\emptyset$ for all $h \in \{t+1,
\dots, r\}$. 
Let $d$ be the linear map with diagonal matrix $d=\diag({C_1,\ldots,C_r})$
where $C_k \in M_{|\Lambda_k|} (\bbK)$ is defined as follows:
\begin{itemize}
    \item If $k \in \{t+1,\dots, r\} $ then $C_k:=0$.
    \item If $k \in \{1,\dots, t\}$ and $\delta (i,\Lambda_k)$ is odd then $C_k=:2 I_{|\Lambda_k|}$.
    \item If $k \in \{1,\dots, t\}$, $\delta (i,\Lambda_k)$ is even and $\Lambda_k=\{k_1, \dots, k_s\}$ then
    $$C_k:=\left( \begin{array}{ccccc}
    1&0&\dots&0&-3\alpha_{k_sk_1}^2  \\
    0&1&\dots&0&-3\alpha_{k_sk_2}^2  \\
    \vdots&\vdots&\ddots&\vdots &\vdots \\
    0&0&\dots &1&-3\alpha_{k_sk_{s-1}}^2 \\
    3\alpha_{k_sk_1}&3\alpha_{k_sk_2}&\dots &3\alpha_{k_sk_{s-1}}&1
    \end{array} \right).$$
\end{itemize}
To prove that $d=(d_{ij}) \in \Der(\A)$ it is sufficient to verify that $d$ satisfies the conditions of the Proposition \ref{pro:dervolt}. It is clear that the condition  \ref{pro:dervolt1} of Proposition \ref{pro:dervolt} holds.

Now we are going to verify that $d$ satisfies the condition \ref{pro:dervolt2} of Proposition \ref{pro:dervolt}. Let $\ell,j \in \Lambda$, $\ell \neq j$ and $ \{\ell, j\} \subseteq \Lambda_k$. If $k \in \{t+1,\dots,r \}$ or $k \in \{1,\dots,t \}$ and $\delta(i,\Lambda_k)$ is odd, then $d_{\ell j}=d_{j \ell}=0$ by the definition of $d$. Now, observe that if   $\{ \ell ,j \} \subseteq \Lambda_k\setminus \{k_s\}$ where $\Lambda_k=\{k_1, \dots, k_s \}$ and $\delta(i,\Lambda_k)$ is even then  $d_{\ell j}=d_{j\ell}=0$. Finally, if $\{ \ell,j \}\cap \{k_s\}\neq \emptyset$ we can assume without loss of generality that $\ell=k_s$. Then, $$d_{\ell j}=d_{k_sj}=3\alpha_{k_sj}=3\alpha_{k_sj}\alpha_{k_sj}\alpha_{jk_s}=\alpha_{jk_s} 3\alpha_{k_sj}^2=-\alpha_{j\ell}d_{j\ell}.$$

Next, we will show that $d$ verifies the condition \ref{pro:dervolt3}  of Proposition  \ref{pro:dervolt}. First, note that if $k \in \{1,\dots, t\}$ and $\delta(i,\Lambda_k)$ is odd, then $\delta (i, \Lambda_{h})$ is even, for all $\Lambda_{h}$ satisfying $\delta (\Lambda_{h},\Lambda_k)=1$. Indeed, if we suppose that $\delta (i, \Lambda_{k})$  and $\delta (i,\Lambda_{h})$ are odd for some $\Lambda_{h}$ with $\delta (\Lambda_{h},\Lambda_k)=1$ then there are paths
$\mu$ from $i$ to  $j_1$ and $\sigma$ from $i$ to $j_2$, where $j_1\in \Lambda_k$ and $j_2 \in \Lambda_{h}$ such that $|\mu|$ and $|\sigma|$ are odd. As $\delta (j_1, j_2)=1$ then we have a cycle  of length $|\mu|+|\sigma|+1$ which is odd, a contradiction. Analogously, if $k \in \{1,\dots, t\}$ and $\delta(i,\Lambda_k)$ is even, then $\delta (i, \Lambda_{h})$ is odd for all $\Lambda_{h}$ verifying $\delta (\Lambda_{h},\Lambda_k)=1$.
Now, let $\ell, j \in \Lambda$ with $\ell \in D^1(j)$. Observe that $\Lambda(\ell)\cap\Lambda(j)=\emptyset$ because $\A$ is a Volterra evolution algebra and so $\omega_{jj}=0$.
We will distinguish several cases:

\noindent \textbf{Case 1.} If $ \ell \in \Lambda_1\cup  \ldots \cup \Lambda_t $ and $j \in \Lambda_{t+1} \cup \ldots \cup\Lambda_{r}$ then $\ell \not\in D^1(j)$, which is a contradiction.

\noindent \textbf{Case 2.} If $\ell, j \in \Lambda_{t+1} \cup  \ldots \cup \Lambda_{r}$ then $d_{\ell\ell}=d_{\ell j}=d_{j\ell}=0$ by definition and we get what we wanted.

\noindent \textbf{Case 3.} In this case we suppose that  $\delta (i, \Lambda(\ell))$ is even. Necessarily $\delta (i,\Lambda(j))$ is  odd since $\delta(\Lambda(j),\Lambda(\ell))=1$. Thus $d_{\ell \ell}=1$, $d_{jj}=2$  and $d_{kj}=0$ for all $k \in \Lambda\setminus \{j\}$. Therefore
$$2d_{\ell \ell}=2=\alpha_{jj}d_{jj}=\sum_{k \in  \Lambda(j)} \alpha_{jk}d_{kj}.$$

\noindent \textbf{Case 4.} Assume that  $\delta (i, \Lambda(\ell))$ odd. As in the previous case, observe that $\delta (i,\Lambda(j))$ is even. Thus, we have that $d_{\ell \ell}=2$. Let $\Lambda(j)=\{k_1,  \dots, k_s\}$.  We are going to consider two cases:

\noindent  \textbf{Case 4.1.} If $j\neq k_s$ then $d_{jj}=1$, $d_{k_sj}=3\alpha_{k_sj}$ and $d_{kj}=0$ for $k\in \Lambda\setminus \{j,k_s\}$. Then
$$\sum_{k \in \Lambda(j)} \alpha_{jk}d_{kj}=\alpha_{jj}d_{jj}+\alpha_{jk_s}d_{k_sj}=1+\alpha_{jk_s}3\alpha_{k_sj}=4=2d_{\ell \ell}.$$

\noindent \textbf{Case 4.2.} If $j=k_s$ then $d_{k_sk_s}=1$ and $d_{kk_s}=-3\alpha_{k_sk}^2$ for $k\in \{k_1,  \dots, k_{s-1}\}$. Then
$$\displaystyle \sum_{k \in \Lambda(j)} \alpha_{k_sk}d_{kk_s}= \alpha_{k_sk_s}d_{k_sk_s} - \sum_{\substack{k \in \Lambda(j) \\ k\neq k_s}}\alpha_{k_sk}3\alpha_{k_sk}^2=1  -3 \sum_{\substack{k \in \Lambda(j) \\ k\neq k_s}}\alpha_{k_sk}^3.$$
Since $\displaystyle \sum_{k \in \Lambda(j)}  \alpha_{k_s k}^3=0$  we get that
$ \sum_{\substack{k \in \Lambda(j) \\ k\neq k_s}}  \alpha_{k_sk}^3=-1$. Therefore,
$$\sum_{k \in \Lambda(j)} \alpha_{k_sk}d_{kk_s}=1-3\sum_{\substack{k \in \Lambda(j) \\ k\neq k_s}}\alpha_{k_sk}^3=4=2d_{\ell \ell}.$$
\end{proof}

The condition of having odd length cycles can not be eliminated as shows the following example.

\begin{exa}\rm
Let $\A$ an evolution algebra with $B=\{e_i\}_{i\in \Lambda}$ ($\Lambda=\{1,\ldots,7\}$)  natural basis and multiplication table $e_1^2=e_2-e_3+e_4-e_5+e_6-e_7$, $e_2^2=-e_3^2=-e_1+e_4-e_5$, $e_4^2=-e_5^2=-e_1-e_2+e_3$ and $e_6^2=-e_7^2=-e_1$.
Since that $M_B$ is skew-symmetric, then $\A$ is a Volterra evolution algebra. 
Consider a natural decomposition of $\Lambda=\Lambda_1\cup \Lambda_2\cup\Lambda_3\cup \Lambda_4$ relative to $B$, where $\Lambda_1=\{1\}$, $\Lambda_2=\{2,3\}$, $\Lambda_3=\{4,5\}$ and $\Lambda_4=\{6,7\}$. Note that $7\in \Lambda$ is such that $\sum_{k \in \Lambda_g}\alpha^3_{jk}=0$ for all $j \in \Lambda (j \in \Lambda_g)$ with $\delta(7,j)$ even. Therefore all hypotheses of Theorem \ref{prop:e^4diineq0}  are verified, except that $E_{\A}^B$ has odd length cycles. Furthermore, an easy computation shows that $\Der(\A)=\{0 \}$. 
\end{exa}

\begin{prop} \label{prop:alpha3=0}
Let $\A$ be a non-degenerate Volterra evolution algebra with a natural basis $B=\{e_i\}_{i \in \Lambda }$, $\Lambda=\Lambda_1\cup \ldots \cup \Lambda_r$ a natural decomposition of $\Lambda$ relative to $B$ and structure matrix $M_B=(\w_{ij})$. If there is $i \in \Lambda$ such that $\displaystyle \sum_{k \in \Lambda_{g}}  \alpha_{jk}^3=0$  for all  $j \in D(i)$ where $i\in \Lambda_{g},$
then there exists $d=(d_{ij}) \in \Der(\A)$ such that $d_{kk}\neq 0$ for all $k \in D(i)$.
\end{prop}

\begin{proof} First, we can reorder the natural decomposition in such a way that $\Lambda=\Lambda_1 \cup \ldots \cup \Lambda_t\cup \Lambda_{t+1} \cup \ldots \cup \Lambda_{r}$, with $D(i)=\Lambda_1  \cup \ldots \cup \Lambda_t$. Define $d=(d_{ij})$  as the diagonal matrix $\diag(C_1,\ldots,C_r)$
where $C_k \in M_{|\Lambda_k|} (\bbK)$ is defined as below:
\begin{itemize}
    \item If $k \in \{t+1, \dots, r\}$ then $C_k:=0$.
    \item If $k \in \{1, \dots, t\}$ and $\Lambda_k=\{k_1, \dots, k_s\}$ then
    $$C_k:=\left( \begin{array}{ccccc}
    1&0&\dots&0&-\alpha_{k_sk_1}^2  \\
    0&1&\dots&0&-\alpha_{k_sk_2}^2  \\
    \vdots&\vdots&\ddots&\vdots &\vdots \\
    0&0&\dots &1&-\alpha_{k_sk_{s-1}}^2 \\
    \alpha_{k_sk_1}&\alpha_{k_sk_2}&\dots &\alpha_{k_sk_{s-1}}&1
    \end{array} \right).$$
\end{itemize}

In order to see that $d$ is a derivation, we will prove that it satisfies the conditions \ref{pro:dervolt1}-\ref{pro:dervolt3} of Proposition \ref{pro:dervolt}. Note that $d$ verifies Proposition \ref{pro:dervolt} \ref{pro:dervolt1} by definition.

Let $p,q \in \Lambda$, $p \neq q$ and $\{p,q\} \subseteq \Lambda_{h}$ for  $h \in \{1,\dots, r\}$. If $h > t$ then $d_{p q}=d_{q p }=0$.
If $h \leq t $, we consider two cases.

\noindent \textbf{Case 1.} If $\{p,q \} \subseteq \Lambda_{h}\setminus \{k_s\}$ then $d_{p q}=d_{qp}=0$.

\noindent \textbf{Case 2.} If $\{p,q \} \cap \{k_s\} \neq \emptyset$, assume, without loss of generality, that $p=k_s$.
 Then $$d_{p q}=d_{k_sq}=\alpha_{k_sq}=\alpha_{k_sq}(\alpha_{k_sq}\alpha_{qk_s})=\alpha_{qk_s} (\alpha_{k_sq}^2)=\alpha_{qk_s} (-d_{qk_s})=-\alpha_{qp}d_{qp}.$$

Now, let $p,q \in \Lambda$, $q \in \Lambda(q)=\Lambda_{h}$ for some $h \in \{1,\ldots, r\}$ and $p \in D^1(q)$.

If $h > t$, then $d_{p p }=d_{kq}=0$ for all $k \in \Lambda_{h}$ since $p\notin D(i)$ otherwise $q \in D(i)$ contradiction. So the condition \ref{pro:dervolt2} of Proposition \ref{pro:dervolt} is verified.

If $h \leq t$  therefore $d_{qq}=1$. Moreover, $p\in D(i)$ then $d_{pp}=1$. Writing $\Lambda_{h} =\{k_1,k_2,\dots, k_s\}$, we will distinguish two cases.

\noindent \textbf{Case 1.} If $q\neq k_s$ then $$\sum_{k \in \Lambda_{h}} \alpha_{qk}d_{kq}=d_{qq}+\alpha_{qk_s}d_{k_sq}=d_{qq}+\alpha_{qk_s}\alpha_{k_sq}=2.$$

\noindent \textbf{Case 2.} If $q = k_s$ then
$$\sum_{k \in \Lambda_{h}} \alpha_{k_sk}d_{kk_s}=
1+\sum_{\substack{k \in \Lambda_{h} \\ k \neq  k_s}} \alpha_{k_sk}(-\alpha_{k_sk}^2)=
1-\sum_{\substack{k \in \Lambda_{h} \\ k \neq  k_s}} \alpha_{k_sk}^3=2.$$

In both cases, we have that $
    2d_{pp}=\sum_{k \in \Lambda_{h}} \alpha_{qk}d_{kq}.$
\end{proof}

\section{The loops of an evolution algebra}\label{loops}

By analogy with graph theory, we define when an element of a natural basis is called loop and then we start by studying what properties of this set is invariant under change of natural basis.

\begin{defn} \label{loop}\rm
Let $\A$ be an evolution algebra with a natural basis $B=\{e_i\}_{i\in \Lambda}$ and structure matrix $M_B=(\omega_{ij})$. We say that $e_i$ is a \emph{loop relative to the basis} $B$ if $\omega_{ii}\not=0$. Otherwise, we say that $e_i$ is a \emph{no-loop}. We denote by $\loops(\A,B)$ the set of loops and by $\nloops(\A,B):=B\setminus \loops(\A,B)$ the set of no-loops.
\end{defn}

\begin{remark}\label{equi_loops} \rm
If $\A$ is an evolution algebra with a natural basis $B=\{e_i\}_{i\in \Lambda}$ then the following conditions are equivalent:
\begin{itemize}
\item $\loops(\A,B)=\emptyset.$
\item $i \notin \supp(e^{2}_i)$ for all $i \in \Lambda$.
\item $e_ie^{2}_i=0$ for all $i\in \Lambda$.
\end{itemize}
\end{remark}

\begin{theorem} \label{ceros}
Let $\A$ be an evolution algebra with natural basis  $B$ and $B'$. Let $B=B_0\cup \ldots \cup B_r$ and $B'=B'_0\cup \ldots \cup B'_r$ be natural decomposition of  $B$ and $B'$ where $B'$ is reordered in such a way that $B'_0=B_0$ and $B'_t  \subseteq \spa( B_0 \cup B_t)$. Then $B_t \subseteq \nloops(\A,B)$ if and only if $B'_t \subseteq \nloops(\A,B')$.
\end{theorem} 

\begin{proof} Let $B=\{e_i\}_{i\in \Lambda}$ and $B'=\{f_i\}_{i\in \Lambda}$ and $M_B=(\omega_{ij})$ and $M_{B'}=(\omega'_{ij})$ the corresponding structure matrices. For $B_0$ the affirmation is trivial. Let $B_t$ with $t\not=0$ such that $B_t \subseteq \nloops(\A,B)$. Suppose, contrary to our claim, that there exists $f_{j}\in B'_t \cap \loops(\A,B')$. Then according to Remark \ref{equi_loops} we have that  $f_{j}^2f_{j}=\omega_{jj}'f_{j}^2\neq 0.$ Let $e_{\ell}\in B_t$ and $\alpha_{{{\ell}}k} \in \mathbb{K}^{\times}$ for $k\in \Lambda_t$ such that $e_{k}^2=\alpha_{{{\ell}}k}e_{{\ell}}^2$. Then $D^1(k)=D^1({\ell})$, for $k\in \Lambda_t$. On the other hand, using Remarks \ref{decom} and \ref{unicidaddes} we can write $f_{j}=\sum_{k \in \Lambda_t}x_ke_k+\sum_{k \in \Lambda_0}x_ke_k$. Then  
$$ f_{j}^2=\sum_{k \in\Lambda_t}x_k^2e_k^2
=\sum_{k \in\Lambda_t}x_k^2\alpha_{{\ell}k}e_{{\ell}}^2=\beta \sum_{p \in D^1({\ell})} \omega_{{\ell}p}e_p,$$
where $\beta=\sum_{k \in\Lambda_t}x_k^2\alpha_{{\ell}k}$.
Therefore
$$
0\neq f_{j}^2f_{j}=f_{j}^2\left(\sum_{k \in \Lambda_t}x_ke_k+\sum_{k \in \Lambda_0}x_ke_k \right)=f_{j}^2\sum_{k \in \Lambda_t}x_ke_k
=  \beta \left(\sum_{p \in D^1({\ell})} \omega_{{\ell}p}e_p\right) \left(\sum_{k \in \Lambda_t}x_ke_k \right).
$$
Thus there exists $e_{q} \in B_t$ such that $q \in D^1(\ell)=D^1(q)$, that is, $e_{q} \in \loops(\A,B)$, contrary to our assumption. The proof of the reciprocal statement is analogous.
\end{proof}

\begin{corollary}\label{coro:loopsperfect}
Let $\A$ be an evolution algebra. Assume that there exists a natural basis $B$ of $\A$ satisfying some of the following conditions:
\begin{enumerate}[label=(\roman*)]
    \item $\A$ has Property $\rm{(2LI)}$.
    \item $\loops(\A,B)=\emptyset$. 
    \item $\A^2=\A$.
    \item $\A$ is twin-free relative to $B$. 
    \item $\A$ is a Volterra evolution algebra relative to $B$.
\end{enumerate}
Then $|\loops(\A,B)|=|\loops(\A,B')|$ for all natural basis $B'$.
\end{corollary}

\begin{corollary}\label{coro:loopsperfect2}
Let $\A$ be an evolution algebra with natural basis $B$ and $B'$. Then $\loops(\A,B)=\emptyset$ if and only if $\loops(\A,B')=\emptyset$. 
\end{corollary}

However, the number of loops is not an invariant under the natural base change as the following example shows.

\begin{exa} \rm
Consider the evolution algebra $\A$ with  natural basis $\{e_i\}_{i=1}^3$ such that  $e_1^2=e_2^2=e_2+e_3$ and $e_3^2=e_1-e_2+e_3$. Observe that $\loops(\A,B)= \{ e_2,e_3\}$. If we take another natural basis $B'=\{f_i\}_{i=1}^3$ with $f_1=e_1+e_2$, $f_2=e_3$ and $f_3=e_1-e_2$ then $\loops (\A,B')=B'$.
\end{exa}

\begin{theorem}\label{theo:invariance_of_loops}
Let $\A$ be an evolution algebra with a natural basis $B=\{e_i\}_{i\in \Lambda}$ and  $\Lambda=\Lambda_0\cup \Lambda_1\cup\ldots \cup \Lambda_r$ be a natural decomposition of $\Lambda$ relative to $B$. If for all $t\in  \{1,\ldots,r\}$ such that $|\Lambda_t|>1$ we have that  $B_t \subseteq \nloops(\A,B)$   
then $\vert \loops(\A,B)\vert = \vert \loops(\A,B') \vert$ for any natural basis $B'$ of $\A$.
\end{theorem}
\begin{proof}Let $B'=\{f_i\}_{i\in \Lambda}$ be a natural basis of $\A$. We define $\Lambda^1:= \bigcup_{\vert \Lambda_t \vert =1} \Lambda_t$.
By our assumption, $\loops(\A,B) \subseteq \{e_i \colon i \in \Lambda^1\}$. Note that if $i \in \Lambda^1$ then, by Theorem \ref{ceros}, up to reordering, $e_i \in \loops(\A,B)$ if and only if $f_i \in \loops(\A,B')$. Therefore $\vert \loops(\A,B) \vert =\vert \loops(\A,B') \vert $.
\end{proof}

\begin{theorem}\label{theo:notinvariance_of_loops}
Let $\A$ be an evolution algebra with a natural basis  $B=\{e_i\}_{i\in \Lambda}$. Consider a natural decomposition $B=B_0\cup B_1\cup \ldots \cup B_r$ and 
suppose that there exists  $t \in \{1,\dots, r\}$ such that
$|B_t|>1$ and  satisfies some of the following conditions
\begin{enumerate}[label=(\roman*)]
    \item $B_t\cap \loops(\A,B)\neq \emptyset$ and $B_t\cap \nloops(\A,B)\neq \emptyset$. \label{theo:notinvariance_of_loops1}
    \item $B_t \subseteq \loops(\A,B)$ and there exist $e_{q},e_{p}\in B_t$ such that $\alpha_{qp}\neq- \left(\frac{\omega_{qq}}{\omega_{qp}}\right)^2$ where $e_p^2=\alpha_{qp}e_q^2$.\label{theo:notinvariance_of_loops2}
    \item $B_t \subseteq \loops(\A,B)$ and $|B_t|>2$. \label{theo:notinvariance_of_loops3}
\end{enumerate}
Then there is a natural basis $B'$ such that $|\loops(\A,B)|\neq |\loops(\A,B')|$.
\end{theorem}

\begin{proof}
First, assume that $B_t$ satisfies  
\ref{theo:notinvariance_of_loops1}. Let $e_{i} \in B_t\cap \loops(\A,B) $ and $e_{j} \in B_t\cap \nloops(\A,B)$. Let $\alpha_{ij}\in \mathbb{K}^{\times}$ such that $e_{j}^2=\alpha_{ij}e_{i}^2$ and  $\gamma \in \mathbb{K}\setminus\{0,1\}$ such that $\gamma^2\neq \frac{-1}{\alpha_{ij}}$. Then $\omega_{i j}=\frac{1}{\alpha_{ij}}\omega_{j}=0$. Consider the basis $B'=\{f_k\}_{k\in \Lambda}$ where
$$f_k=\left\{ \begin{array}{ll}
     e_k,& \text{if } k \neq i, j,  \\
     e_{i}+\gamma e_{j},&\text{if } k=i,  \\
     -\gamma \alpha_{ij} e_{i}+e_{j},&\text{if } k=j.
\end{array}\right.$$

Note that $f_{i}f_{j}=-\gamma \alpha_{ij}e_{i}^2+\gamma e_{j}^2=-\gamma \alpha_{ij}e_{i}^2+\gamma \alpha_{ij}e_{i}^2=0.$
So, $B'$ is a natural basis for $\A$. Observe now that $f_kf_k^2=0$ if and only if $e_ke_k^2=0$ for all $k \in \Lambda\setminus\{i, j\}$. 
Now we will prove that  $f_i f_i^2 \neq 0$ and $f_jf_j^2 \neq 0$. Indeed,
$$
e_{i}e_{i}^2=\omega_{ii}e^{2}_{i},\phantom{a} e_{j}e_{j}^2=e_{j}e_{i}^2=0, f_{i}^2=(1+\alpha_{ij}\gamma^2)e^{2}_{i}\text{ and } f_{j}^2=(\alpha_{ij}+\alpha_{ij}^2\gamma^2)e^{2}_{i}.
$$
Then
$$
\begin{array}{lclclcl}
f_{i}f_{i}^2&=&(e_{i}+\gamma e_{j})(1+\alpha_{ij}\gamma^2)e_{i}^2 =(1+\alpha_{ij}\gamma^2)\omega_{i i}e_{i}^2 \not= 0\\ [0.2cm]
f_{j}f_{j}^2 &=&(-\gamma \alpha_{ij} e_{i}+e_{j})(\alpha_{ij}+\alpha_{ij}^2\gamma^2)e^{2}_{i}={-\gamma \alpha_{ij}}(\alpha_{ij}+\alpha_{ij}^2\gamma^2) \omega_{i i}e^{2}_{i}\not=0
\end{array}
$$
Therefore   $|\loops(\A,B)|=|\loops(\A,B')|+1.$

Now, assume that $B_t$ satisfies \ref{theo:notinvariance_of_loops2}. Define $\gamma := -\alpha_{qp}\frac{\omega_{qp}}{\omega_{qq}}$ and $\beta:=\frac{\omega_{qp}}{\omega_{qq}}$. Consider the set $B'=\{f_k\}_{k\in \Lambda}$ where
$$f_k=\left\{ \begin{array}{ll}
     e_k,& \text{if } k \neq q,p,  \\
     \gamma e_{q}+ e_{p},&\text{if } k=q,  \\
      e_{q}+\beta e_{p},&\text{if } k=p.
\end{array}\right.$$

By hypotheses $\gamma \beta\neq 1$ then $B'$ is a basis for $\A$. Similarly to item \ref{theo:notinvariance_of_loops1},  $f_kf_k^2=0$ if and only if $e_ke_k^2=0$ for all $k \in \Lambda\setminus \{q,p\}$. Therefore

$$f_{q}f_{q}^2=(\gamma e_{q}+ e_{p})(\gamma^2+\alpha_{qp})e_{q}^2
=(\gamma^2+\alpha_{qp})(\gamma \omega_{qq} e_{q}^2+\omega_{qp}e_{p}^2) 
=(\gamma^2+\alpha_{qp})(\gamma \omega_{qq}+\omega_{qp}\alpha_{qp})e_{q}^2=0.$$
So  $|\loops (\A,B')|<|\loops (\A,B)|$.

\vspace{0.2cm} Finally, assume that $B_t$ satisfies \ref{theo:notinvariance_of_loops3}. Suppose that  $\alpha_{ji}=-\left(\frac{\omega_{jj}}{\omega_{ji}}\right)^2$ for all $j, i \in \Lambda_t$. Then we will get a contradiction. Indeed, let $q,p,\ell \in \Lambda_t$. Since $\alpha_{qp}=-\frac{\omega_{qq}^2}{\omega_{qp}^2}$ then $\omega_{pq}=-\frac{\omega_{qq}^3}{\omega_{qp}^2}$ and analogously $\omega_{\ell q}=-\frac{\omega_{qq}^3}{\omega_{qt}^2}$. Consequently $\omega_{pp}=-\frac{\omega_{qq}^2}{\omega_{qp}}$ and $\omega_{p\ell}=-\omega_{q\ell}\frac{\omega_{qq}^2}{\omega_{qp}^2}$. Moreover, using that $\alpha_{p\ell}=-\frac{\omega_{pp}^2}{\omega_{p\ell}^2}$ we get that $\omega_{tq}=\alpha_{pt}\w_{pq}=\frac{\omega_{qq}^3}{\omega_{q\ell}^2}$, which is a contradiction. Therefore there exist $q,p \in \Lambda_t$ such that $\alpha_{qp}\neq- \frac{\omega_{qq}^2}{\omega_{qp}^2}$ and the affirmation follows from item \ref{theo:notinvariance_of_loops2}.
\end{proof}

\begin{prop}\label{B_i: pequeno}
Let $\A$ be a non-degenerate evolution algebra and $B=\{e_i\}_{i\in \Lambda}$ a natural basis of $\A$. Consider a natural decomposition $B=B_1\cup \ldots \cup B_r$ and suppose that $|B_s|\leq 2$ for all $s \in \{1,\dots, r \}$  such that $B_s\cap L(\A,B)\neq \emptyset$.
If for every $B_t$, with $|B_t|=2$, some of the following conditions is satisfied:
\begin{enumerate}[label=(\roman*)]
    \item $B_t \subseteq \nloops (\A,B)$, 
    \item $B_t \subseteq \loops(\A,B)$ and $\alpha_{jk}=-\left(\frac{\omega_{jj}}{\omega_{jk}}\right)^2$ for all $j, k \in \Lambda_t$ where $e_k^2=\alpha_{jk}e_j^2$,
\end{enumerate}
then $|\loops (\A,B)|=|\loops(\A,B')|$ for all natural basis $B'$ of $\A$.
\end{prop}

\begin{proof} Let $\Lambda= \Lambda_1\cup\ldots \cup \Lambda_r$ be a natural decomposition of $\Lambda$ relative to $B$. Let us define 
$$\Lambda^1:= \bigcup_{\vert \Lambda_t \vert =1}\Lambda_t, \phantom{aaa} \Lambda^2:= \bigcup_{\substack{| \Lambda_t | =2 \\ B_t \subseteq \loops(\A,B)}} \Lambda_t \phantom{aa} \text{ and }  \phantom{aa}\Lambda^3:= \bigcup_{\substack{| \Lambda_t | \geq 2 \\ B_t \subseteq \nloops(\A,B)}} \Lambda_t.$$
Suppose, contrary to our claim, that there is a natural basis $B'=\{f_k\}_{k \in \Lambda}$ of $\A$ such that $|\loops(\A,B')|\neq |\loops(\A,B)|$. Let   $B'=B'_1\cup \ldots \cup B'_r$ be a natural decomposition, ordered in such a way that $B'_t \subseteq \spa (B_t)$  for any $t$ and let $M_{B'}=(\omega_{kj}')$ be the structure matrix of $\A$ relative to $B'$.  By Theorem \ref{ceros} we have that $\big|\{ k \in \Lambda^1 :\omega_{kk}'\neq 0\} \big|= \big|\{ k \in \Lambda^1 :\omega_{kk}\neq 0\} \big|$ and $\omega_{jj}'=0$ for all $j \in \Lambda^3$. Therefore there exists $h\in \{1,\dots, r\} $ such that 
$B_h=\{e_{i},e_{\ell}\}\subseteq \loops (\A,B)$ and  $B_h'=\{f_{k},f_{j}\}\not\subseteq \loops (\A,B')$. Without loss of generality, we assume that  $f_{k}f_{k}^2= 0$. Then we have that
$$ f_{k}=x_{11}e_{i}+x_{12}e_{\ell}
\text{ and }
f_{j}=x_{21}e_{i}+x_{22}e_{\ell},
$$
where $x_{ij} \in \bbK $. Consequently, using that $e_{\ell}^2=\alpha_{i\ell}e^2_{i}$, $e_{i}e_{i}^2=\omega_{ii}e^2_{i}$ and $e^{2}_{i}e_{\ell}=\omega_{i\ell}e_{\ell}^2$  it follows that
\begin{eqnarray}\label{eq:f1f12}
f_{k}f_{k}^2&=&(x_{11}e_{i}+x_{12}e_{\ell})(x_{11}^2+x_{12}^2\alpha_{i\ell})e_{i}^2 \nonumber \\
&=&(x_{11}^2+x_{12}^2\alpha_{i\ell})(x_{11}\omega_{ii}e_{i}^2+x_{12}\omega_{i\ell}e_{\ell}^2)\nonumber \\
&=&(x_{11}^2+x_{12}^2\alpha_{i\ell})(x_{11}\omega_{ii}+x_{12}\omega_{i\ell}\alpha_{i\ell})e_{i}^2=0.
\end{eqnarray}
As $f_{k}^2\neq 0$ then $x_{11}^2+x_{12}^2\alpha_{i\ell}\neq 0 $. Thus
$$x_{11}\omega_{ii}-x_{12}\omega_{i\ell}\frac{\omega_{ii}^2}{\omega_{i\ell}^2}=\omega_{ii}\left(x_{11}-x_{12}\frac{\omega_{ii}}{\omega_{i\ell}}\right)=0,$$
where we use that, by hypotheses, $\alpha_{i\ell}=-\frac{\omega_{ii}^2}{\omega_{i\ell}^2}$.  The fact that  $\w_{ii} \neq 0$ implies $x_{11}=\frac{\omega_{ii}}{\omega_{i\ell}}x_{12}$. Therefore $x_{11}x_{12} \neq 0$ because $f_k \in B^{\prime}$. On the other hand, since $B'$ is a natural basis 
\begin{eqnarray}\label{eq:orto}
f_{k}f_{j} =x_{11}x_{21}e_{i}^2+x_{12}x_{22}e_{\ell}^2=(x_{11}x_{21}+x_{12}x_{22}\alpha_{i\ell})e_{i}^2=0.
\end{eqnarray}
Now, by Equation \eqref{eq:orto} we obtain that
$$x_{21}=\frac{x_{12}x_{22}}{x_{11}}\frac{\omega_{ii}^2}{\omega_{i\ell}^2}=x_{12}x_{22}\frac{\omega_{ii}^2}{\omega_{i\ell}^2}\frac{\omega_{i\ell}}{x_{12}\omega_{ii}}=\frac{\omega_{ii}}{\omega_{i\ell}}x_{22}.$$
Finally, we conclude that 
$$x_{11}x_{22}-x_{12}x_{21}=x_{12}\frac{\omega_{ii}}{\omega_{i\ell}}x_{22}-x_{12}\frac{\omega_{ii}}{\omega_{i\ell}}x_{22}=0.$$
Therefore $f_{k}$ and $f_{j}$ are linearly dependent, which is a contradiction. 
\end{proof}

\begin{prop}\label{ceros_omega}
Let $\A$ be an evolution algebra with a natural basis $B=\{e_i\}_{i \in \Lambda}$  such that the structure matrix $M_B=(\omega_{ij})$ satisfies $\omega_{ij}=0$ if and only if $\omega_{ji}=0$, for all $i,j \in \Lambda$, $i\not =j$. Consider the natural decomposition $\Lambda=\Lambda_0\cup \ldots \cup \Lambda_r$ relative to $B$. Suppose that $e_i \in \nloops(\A,B)$. Then $e_j \in \nloops(\A,B)$ for every $j \in \Lambda{(i)}$.
\end{prop} 

\begin{proof} 
Since $\omega_{ii}=0$ and $\rk(\{e_i^2,e_j^2\})=1$  we get that $\omega_{ji}=0$. By symmetry $\omega_{ij}=0$. Again, as $\rk(\{e_i^2,e_j^2\})=1$ then $\omega_{jj}=0$.
\end{proof}

\begin{remark} \rm
If $A$ is an evolution algebra and $B$ is a natural basis satisfying the assumptions of the previous proposition and $B=B_0 \cup \ldots \cup B_r $ is a natural decomposition then either $B_t \subseteq \loops(\A,B)$ or $B_t \subseteq \nloops(\A,B)$ for every $t \in \{0,1,\ldots,r\}$.  

\end{remark}

\begin{corollary} \label{recopi}Let $\A$ be a non-degenerate evolution algebra with a natural basis  $B=\{e_i\}_{i\in \Lambda}$. Consider a natural decomposition $B= B_1\cup \ldots \cup B_r$. Then the number of loops in $\A$ is invariant by change of natural basis if $|B_t|\leq 2$ for all $t \in \{1,\dots,r \}$ such that $B_t\cap \loops(\A,B)\neq \emptyset$ and for every $s$ such that $|B_s|=2$, some of the following conditions is satisfied:
\begin{enumerate}[label=(\roman*)]
    \item $B_s \subseteq \nloops(\A,B),$
    \item $B_s \subseteq \loops(\A,B)$ and $\alpha_{jk}=-\left(\omega_{jj} /\omega_{jk}\right)^2$ for all $j, k \in \Lambda_s$ where $e_k^2=\alpha_{jk}e_j^2$.
\end{enumerate}
\noindent Otherwise the number of loops of $\A$ depend of the natural basis.
\end{corollary}

\begin{proof} The statement that the number of loops is invariant by the change of base follows from Proposition \ref{B_i: pequeno}. On the other hand, by Theorem \ref{theo:notinvariance_of_loops}, in any other case the number of loops in $\A$ depends on the natural basis.
\end{proof}

\end{document}